\documentclass[11pt]{article}
\usepackage[left=3.2cm,right=3.2cm,top=3.5cm,bottom=3cm]{geometry}
\usepackage[utf8]{inputenc}
\usepackage[english]{babel}
\usepackage{amsmath}
\usepackage{amssymb}
\usepackage{amsthm}
\usepackage{newpxtext,newpxmath}
\usepackage{braket}
\usepackage{amscd}
\usepackage{setspace}
\usepackage{booktabs}
\usepackage{lscape}
\usepackage{rotating}
\usepackage{graphicx}
\usepackage{tikz}
\usepackage{xcolor}
\usepackage{mathtools}
\usepackage{sidecap}
\usepackage{emptypage}
\usepackage{enumitem}
\usepackage{todonotes}
\newtheorem{theorem}{Theorem}
\newtheorem*{theorem*}{Theorem}

\newtheorem{lemma}{Lemma}
\newtheorem{proposition}{Proposition}
\newtheorem{corollary}{Corollary}
\newtheorem*{corollary*}{Corollary}

\newtheorem*{conjecture*}{Conjecture}

\newtheorem*{result*}{Result}
\theoremstyle{remark}
\newtheorem{remark}{Remark}
\theoremstyle{definition}

\newtheorem{example}{Example}
\newtheorem*{example*}{Example}

\newtheorem*{Ack}{Acknowledments}

\newcommand{\Z}{\mathbb{Z}}
\newcommand{\C}{\mathbb{C}}
\newcommand{\Q}{\mathbb{Q}}
\newcommand{\R}{\mathbb{R}}
\newcommand{\F}{\mathbb{F}}
\newcommand{\pro}{\mathbb{P}}
\newcommand{\Po}{\mathbb{H}}
\newcommand{\OO}{\mathcal{O}}
\newcommand{\fa}{\mathfrak{a}}

\newcommand{\fp}{\mathfrak{p}}

\newcommand{\Spec}{\mathrm{Spec}}

\newcommand{\SL}{\mathrm{SL}}

\newcommand{\ord}{\mathrm{ord}}

\newcommand{\A}{\mathcal{A}}
\newcommand{\V}{\mathbb{V}}
\newcommand{\Lo}{\mathbb{L}}

\def\={\;=\;}
\def\:={\;:=\;}
\def\+{\;+\;}
\def\-{\;-\;}
\title{Integrality of Picard-Fuchs differential equations of Kobayashi geodesics and applications}
\author{Gabriele Bogo}
\date{}
\begin{document}
\maketitle 

\begin{abstract}
We prove that the holomorphic solutions of Picard–Fuchs differential equations associated with 1-parameter families of abelian varieties with real multiplication admit power series expansions with $S$-integral coefficients at a maximal unipotent monodromy point. This extends classical integrality results for hypergeometric functions and Bouw–Möller’s work on Teichmüller curves. The integral solutions are related to the non-ordinary locus of the modulo~$p$ reduction of the family, whose cardinality we bound in terms of the Euler characteristic and Lyapunov exponents of the base curve. In some cases, the non-ordinary locus can be recovered by truncating the integral solutions, as in Igusa’s classical observation for the Legendre family. We also establish $S$-integrality of expansions of modular forms at cusps in terms of a modular function for (not necessarily arithmetic) Fuchsian groups with modular embeddings, and deduce congruences. These results are applied in~\cite{BL2} to construct lifts of partial Hasse invariants for rational curves in Hilbert modular varieties.
\end{abstract}

Interest in linear differential equations on complex domains that admit power series solutions with integral expansion coefficients spread after R.\,Apéry's proof in 1978 of the irrationality of the zeta value~$\zeta(3)$.
Following this discovery, D. Zagier conducted an extensive numerical exploration~\cite{ZagierA}, investigating for which quadratic polynomials~$P(t)$ and constant~$\lambda$, called accessory parameter, the differential equation
\begin{equation}
\label{eq:diff1}
\bigl(tP(t)y'(t)\bigr)'+(t-\lambda)y(t) = 0
\end{equation}
admits a solution at~$t = 0$ of the form~$y(t) = 1+\sum_{n=1}^\infty y_n t^n$ with~$y_n \in \mathbb{Z}$ for all~$n \ge 1$.
Zagier found essentially only six pairs $(P(t),\lambda)$ with this property. In all these cases, the solution~$y(t)$ is~\emph{modular}, that is, a modular form~$f$ expressed locally in terms of a modular function~$t$ (a Hauptmodul in Zagier's cases) for the monodromy group~$\Gamma\subset\SL_2(\R)$. In other words, these differential equations are~\emph{uniformizing differential equations} of the Riemann surface~$\Po/\Gamma$ in the sense of Poincaré (see Section~\ref{sec:unif} for more details).
Based on his computer search, Zagier proposed the following still open conjecture. 
\begin{conjecture*}[Zagier~\cite{ZagierA}]
Any integral solution of the differential equation~\eqref{eq:diff1}, where~$P(t)$ is a non-degenerate quadratic polynomial, is modular. 
\end{conjecture*}

In~\cite{BM}, Bouw and Möller disproved a related conjecture of D. and~G.\,Chudnovsky, by presenting an explicit differential equation with integral solution in~$\mathcal{O}_{\Q(\sqrt{13})}[\frac{1}{2}][\![t]\!]$ of shape similar to~\eqref{eq:diff1} with~$P(t)$ of degree three. Their example is particularly interesting because, while not discussing it directly, it sheds new light on Zagier's conjecture. The monodromy of the Bouw-Möller differential equation is a~\emph{non-arithmetic} Fuchsian group, and the solution of their differential equation is both integral and modular. 
Although not explicitly stated, the integrality in Zagier's conjecture should arise from the integrality of the~$q$-expansion of the modular solution, a property not generally expected for modular forms on non-arithmetic groups. Consequently, one may still expect Zagier's conjecture to hold, but in a wider sense, including examples beyond congruence (or arithmetic) Fuchsian groups (from now on we consider Zagier's question for polynomials~$P(t)$ in~\eqref{eq:diff1} of any degree, and differential operators of more general form, see~Section~\ref{sec:unif}). 
The Bouw-Möller example was revisited by Möller and Zagier in~\cite{MZ}, where the surprising integrality was re-proved by embedding the non-arithmetic monodromy group into a Hilbert modular group (a procedure known as~\emph{modular embedding}) and, roughly, by realizing the modular solution as a restriction of a Hilbert modular form with integral coefficients. 

The Bouw-Möller differential equation, as well as those considered by Zagier, are~\emph{Picard-Fuchs differential equations}, meaning that their holomorphic solution is a period function of a family of algebraic varieties varying over the curve~$\Po/\Gamma$. In Zagier's case the varieties are elliptic curves, and in the Bouw-Möller case they are abelian surfaces. While uniformizing differential equations are notoriously difficult to study (one must solve the~\emph{accessory parameter problem}, see~\cite{B1} and~\cite{B2}), Picard-Fuchs differential equations are, due to their geometric origin, easier to understand. 
This motivates the study of second-order differential equations on~$\pro^1(\C)$ that are simultaneously uniformizing and arise as Picard–Fuchs equations of more general families of abelian varieties. 
Restricting to affine base curves, a beautiful result of Möller-Viehweg~\cite{MV} shows that this occurs precisely for the uniformizing differential equations of genus zero non-compact curves lying in a Hilbert modular variety. Such curves are called (genus zero) affine~\emph{Kobayashi geodesics}, and another result in~\cite{MV} implies that they admit an integral model over~$\mathcal{O}_K[S^{-1}]$, for~$K$ a number field and~$S$ a finite set of primes. 

Relevant examples of affine Kobayashi geodesics include modular curves and Teichmüller curves, including curves uniformized by triangle groups, whose Picard-Fuchs differential equations are classical hypergeometric equations. 
As in the Bouw-Möller example, the uniformizing Fuchsian group of a Kobayashi geodesic is in general non-arithmetic; that is, it is not obtained via the usual construction from a quaternion algebra.

The first goal of this paper is to prove a result that complements Zagier's conjecture. We prove that any uniformizing differential equation that is also a Picard-Fuchs equation for a family of abelian varieties admits an integral solution. 
\begin{theorem*}[Theorem~\ref{thm:main}]
Let~$L$ be the uniformizing differential operator of a genus zero affine Kobayashi curve defined over~$\mathcal{O}_K[S^{-1}]$, and let~$y(t)=1+\sum_{n=1}^\infty{y_nt^n}$ be the normalized holomorphic solution at a maximal unipotent monodromy point. Then the coefficients~$\{y_n\}_n$ belong to~$\mathcal{O}_K[S^{-1}]$. 
\end{theorem*}
A stronger statement holds. 
As explained in Section~\ref{sec:PF}, a Kobayashi geodesic~$Y$ gives rise to~$g$ second-order Picard–Fuchs equations, where~$g$ is the dimension of the ambient Hilbert modular variety. Theorem~\ref{thm:main} establishes integrality for the solutions of all these differential equations.

The strategy of the proof of Theorem~\ref{thm:main} generalizes the approach of Bouw–Möller~\cite{BM}, which was carried out for explicit differential equations arising from certain Teichmüller curves in Hilbert modular surfaces. It consists in lifting polynomial solutions of differential equations defined modulo~$p^n$, obtained via Katz's theory of expansion coefficients~\cite{K84}.
The proof in the general case requires a detailed study of the differential equations satisfied by modular forms (Section~\ref{sec:modde}), which is of independent interest, as well as an analysis of the geometry of Hilbert modular varieties in positive characteristic (Sections~\ref{sec:HMV} and~\ref{sec:CHMV}).

A first consequence of Theorem~\ref{thm:main} concerns modular forms. We establish an integral structure on the space of modular forms for non-cocompact genus-zero Fuchsian groups with a modular embedding, by considering certain expansions at the cusps.
The main novelty of this integrality result lies in the possible non-arithmetic nature of the Fuchsian groups, whose theory of modular forms is not covered by classical results. Modular forms on non-arithmetic groups are generally not expected to have algebraic Fourier coefficients at the cusps. In some examples, such as triangle curves and certain Teichmüller curves discussed in~\cite{Wo} and~\cite{MZ} respectively, the~$q$-expansion of modular forms at the cusps has a transcendental radius of convergence~$A$, and the Fourier coefficients contain a transcendental factor. Scaling the~$q$-parameter to~$Q = A \cdot q$ yields algebraic coefficients, but they are generally not~$p$-integral for a given prime~$p$.

Fix a cusp and let~$t$ be a modular function with a simple zero at that cusp, which we take as a local parameter. Then every modular form~$f$ admits a local expansion as a power series in~$t$, called the~\emph{$t$-expansion of~$f$} at the cusp.
\begin{theorem*}[Theorem~\ref{thm:intmf}]
Let~$\Gamma$ be the uniformizing group of a genus-zero affine Kobayashi geodesic admitting a model over~$\mathcal{O}_K[S^{-1}]$, and let~$t$ be a suitable Hauptmodul. Then, for every weight~$k\ge0$, the space of modular forms~$M_k(\Gamma)$ admits a basis consisting of forms whose~$t$-expansions are $p$-integral for almost all primes~$p\not\in S$.
\end{theorem*}
The theorem holds more generally for the larger ring of~\emph{twisted modular forms} (definition in Section~\ref{sec:modde}), a natural generalization of modular forms introduced by Möller and Zagier~\cite{MZ}. These are complex-valued holomorphic functions on~$\Po$ with a multi-weight~$(k_1,\dots,k_g)$, whose modular transformation takes into account the modular embedding of~$\Gamma$ in a Hilbert modular group.
Using Theorem~\ref{thm:main} and the description of the differential equations in terms of Rankin–Cohen brackets, we prove that the (meromorphic) modular form~$t'$ has integral~$t$-expansion. In genus zero, all modular forms are algebraic functions of a Hauptmodul and its first derivative; this implies Theorem~\ref{thm:intmf} for classical modular forms.
To prove the full version of Theorem~\ref{thm:intmf}, the structure of the space of twisted modular forms as a ring of Laurent polynomials over~$M_*(\Gamma)$ is determined in Proposition~\ref{prop:gen}, a result of independent interest.

Although the~$t$-expansion of twisted modular forms may seem exotic, it turns out to be quite useful. As a corollary of the above theorems, certain twisted modular forms satisfy congruences.
\begin{corollary*}[Corollary~\ref{cor:cong} and Remark~\ref{rmk:cong}]
Let~$Y = \Po/\Gamma$ be a genus-zero affine Kobayashi geodesic in a~$g$-dimensional Hilbert modular variety with real multiplication field~$F$, admitting an integral model over~$\mathcal{O}_K[S^{-1}]$. Let~$t$ be a suitable Hauptmodul for~$\Gamma$.
For every prime~$p\not\in S$, there exist non-constant twisted modular forms~$f_{p,1},\dots,f_{p,g}$, polynomials~$\alpha_{p,1}(t),\dots,\alpha_{p,g}(t)\in\mathcal{O}_K[S^{-1}][t]$, and a permutation~$\mathscr{l}_p$ of the set~$\{1,\dots,g\}$, such that for every~$j\in\{1,\dots,g\}$
\[
f_{p,j}(t)\equiv\alpha_{p,j}(t)\cdot f_{p,\mathscr{l}_p(j)}(t^p)\mod p
\]
holds for the~$t$-expansion at a cusp. 
\end{corollary*}
The permutation depends only on the splitting of~$p$ in~$F$: for instance, it is the identity if~$p$ splits completely in~$F$. 

These congruences for modular forms arise from congruences between the integral solutions of the Picard–Fuchs differential equations, and are more naturally formulated in that setting. They are new even for Gauss’s hypergeometric functions: concrete examples, related to the Kobayashi curve~$\Po/\Delta(2,5,\infty)$, are 
\[
\begin{aligned}
{}_2F_1\Bigl(\frac{3}{20}, \frac{7}{20};1;t\Bigr)^2&\;\equiv\; (1-t)\;\cdot\;{}_2F_1\Bigl(\frac{1}{20}, \frac{9}{20};1;t^{13}\Bigr)^2\mod 13\\
{}_2F_1\Bigl(\frac{1}{20}, \frac{9}{20};1;t\Bigr)^2&\;\equiv\; (1-t)(3-t)^2\;\cdot\; {}_2F_1\Bigl(\frac{3}{20},\frac{7}{20};1;t^{13}\Bigr)^2\mod 13\,.
\end{aligned}
\]

The corollary and Theorem~\ref{thm:intmf} are used in~\cite{BL2} to construct lifts of~\emph{partial Hasse invariants} for the Kobayashi curve~$Y$, realized as twisted modular forms in characteristic zero. These lifts correspond to the restriction to the modulo~$p$ reduction of~$Y$ of the partial Hasse invariants of the ambient Hilbert modular variety in characteristic~$p$. The results are also applied in~\cite{BL3} to describe the non-ordinary locus of genus-zero affine Kobayashi geodesics in terms of the zeros of orthogonal (Atkin) polynomials, and in terms of Padé approximants of the logarithmic derivative of the integral solutions.

The second application of Theorem~\ref{thm:main} concerns arithmetic aspects of affine Kobayashi curves. The zeros of the polynomials~$\alpha_{p,1}(t),\dots,\alpha_{p,g}(t)$ in the Corollary describe the intersection of the modulo~$p$ reduction of the Kobayashi curve~$Y$ with the components of the non-ordinary locus of the ambient Hilbert modular variety. Computing the degrees of these polynomials therefore amounts to determining the cardinality of the non-ordinary locus of~$Y$ modulo~$p$. We express these degrees in terms of the (orbifold) Euler characteristic~$\chi(Y)$ and certain invariants~$\lambda_1,\dots,\lambda_g\in(0,1]\cap\Q$ of~$Y$ called \emph{Lyapunov exponents}.

\begin{theorem*}[Corollary~\ref{cor:size}]
Let~$Y$ be a genus-zero affine Kobayashi curve with~$r\ge0$ elliptic points in a~$g$-dimensional Hilbert modular variety with real multiplication field~$F$. Consider the integral model of~$Y$ defined over~$\mathcal{O}_K[S^{-1}]$, let~$p\not\in S$ be a prime, and let~$n_p$ be the number of primes of~$\mathcal{O}_K$ over~$p$. 
The cardinality of the non-ordinary locus of~$Y$ modulo~$p$ is bounded by
\[
\Bigl\lfloor{\frac{-\chi(Y)}{2}\cdot\max_{(j,j')}\{p\lambda_{j'}-\lambda_j\}}\Bigr\rfloor\;\le\;\sharp\mathrm{no}_p^Y\;\le\;n_p\cdot\Bigl\lfloor{\frac{-\chi(Y)}{2}\sum_{(j,j')}{(p\lambda_{j'}-\lambda_{j})}}\Bigr\rfloor+r\,,
\]
where the indices~$j$ and~$j'$ are related as in the corollary and~$\lfloor{x}\rfloor$ denotes the integral part of~$x\in\Q$.
\end{theorem*}
For the special case~$g=2$, the result can be refined to give lower and upper bounds to the size of the supersingular locus of~$Y$ modulo~$p$.

Finally, we present a method to compute the non-ordinary locus via truncation of the integral solutions of the Picard–Fuchs equations. This approach recovers Igusa’s observation~\cite{IgLeg} that Deuring’s Hasse polynomial for elliptic curves is a solution modulo~$p$ of the Picard–Fuchs differential equation of the Legendre family.
\begin{theorem*}[Corollary~\ref{cor:calc}, Example~\ref{ex:deltat}]
Let~$Y$ be as above, and let~$y_1,\dots,y_g$ be the integral solutions of the Picard–Fuchs differential equations associated with~$Y$. For a prime~$p\not\in S$, assume that the polynomials~$\alpha_{p,1}(t),\dots,\alpha_{p,g}(t)$ in the corollary satisfy~$\deg(\alpha_{p,j})<p$. Then there exists an integer~$N\ge 1$, depending on the order of the elliptic points of~$Y$, such that the zeros modulo~$p$ of
\[
\mathrm{lcm}\bigl\{[y_1(t)^N]_p,\dots,[y_g(t)^N]_p\bigr\}\quad\text{and}\quad\mathrm{gcd}\bigl\{[y_1(t)^N]_p,\dots,[y_g(t)^N]_p\bigr\}\,,
\]
where~$[X]_m$ denotes truncation at order~$m$, respectively identify the fibers of non-ordinary and superspecial reduction of the integral model over~$Y$.
\end{theorem*}
In Example~\ref{ex:mocu}, we apply Corollary~\ref{cor:calc} to Apery's differential equation for~$\zeta(2)$ and the related family of elliptic curves discovered by Beukers~\cite{BeIrr}. 

For an example related to a non-arithmetic group, consider the family of genus two curves
\begin{equation}
\label{eq:ext}
C_\eta\;:\;\begin{cases}
y^2&=x^5-5x^3+5x-2\eta,\quad \eta\neq\infty\\
y^2&=x^5-1,\quad \eta=\infty\,.
\end{cases}
\end{equation}
The corresponding family of Jacobians lies in the Hilbert modular surface of discriminant~$D=5$, and defines an integral model for the non-arithmetic triangle curve~$\Po/\Delta(2,5,\infty)$. A Hauptmodul for this group with a pole at the cusp is given by~$J=(\eta^2-1)^{-1}$ (it depends on~$\eta^2$ since~$C_\eta\simeq C_{-\eta}$). The integral solutions of the two Picard-Fuchs differential equations are classical hypergeometric functions~$y_j(J)={}_2F_1\Bigl(\frac{5-2j}{20}, \frac{5+2j}{20},1;J^{-1}\Bigr)$ for~$j=1,2$. (see Example~\ref{ex:hgD} for details). Then for~$p>5$ inert in~$\mathbb{Q}(\sqrt{5})$ the components of the family of Jacobians of~\eqref{eq:ext} over~$\bar{\F}_p$ that become supersingular modulo~$p$ correspond to the zeros of the polynomial
\[
\widetilde{\mathrm{ss}}_p(J)\equiv\mathrm{lcm}\biggl(J^{\frac{p-3}{2}}\biggr[{}_2F_1\Bigl(\frac{3}{20}, \frac{7}{20};1;\frac{1}{J}\Bigr)^2\biggl]^5_p,J^{\frac{3p-1}{2}}\biggl[{}_2F_1\Bigl(\frac{1}{20}, \frac{9}{20};1;\frac{1}{J}\Bigr)^2\biggr]^5_p\biggr)\mod p\,.
\]
By defining~$\mathrm{ss}_p(J):=\widetilde{\mathrm{ss}}_p(J)/\gcd(\widetilde{\mathrm{ss}}_p(J),\widetilde{\mathrm{ss}}'_p(J))$ one obtains a polynomial without multiple roots that describes the supersingular locus. 
For instance, if~$p=13$, one finds~$\mathrm{ss}_{13}(J)=J(J-1)(J-9)$, showing that the supersingular fibers of the family~\eqref{eq:ext} over~$\overline{\F}_{13}$ correspond to parameters~$\eta\in\{1,2,\infty\}$.
Interestingly, for~$p=37$, the polynomial is~$\mathrm{ss}_{37}(J)=J(J-1)(J+16)(J^2+J+8)(J^3+22J^2+J+25)$, which shows that there exist supersingular Jacobians~$\mathrm{Jac}(C_\eta)$ over~$\overline{\F}_{37}$ whose~$J$-invariant lies in a cubic extension of~$\F_{37}$ (the same is true for its absolute Igusa invariants). This contrasts with the elliptic case, where supersingular~$j$-values always lie in~$\F_{p^2}$.

In the Appendix, we show how to rewrite the points of supersingular reduction of the family~\eqref{eq:ext} in terms of absolute Igusa invariants, presenting a method that can also be applied to other curves. 

\begin{Ack}
I thank Claudia Alfes, Irene Bouw, Yingkun Li, Matteo Longo, Martin Raum for helpful comments and suggestions. The author is funded by the Deutsche Forschungsgemeinschaft (DFG, German Research Foundation) — SFB-TRR 358/1 2023 — 491392403.
\end{Ack}

\section{Integrality of solutions of Picard-Fuchs differential equations}

\subsection{Picard-Fuchs differential equations for Kobayashi geodesics}
\subsubsection{Kobayashi geodesics}
\label{sec:KG}
Let~$g\ge1$ be an integer and let~$\A_g$ denote a fine moduli space of polarized abelian varieties equipped with a suitable level structure.
Let~$Y$ be an affine complex curve and let~$\V$ be a polarized~$\Z$-variation of Hodge structures over~$Y$. It gives rise to a morphims~$\phi\colon Y\to\mathcal{A}_g$, where~$g=\mathrm{rank}(\V)/2$. 
We denote by~$f\colon\mathcal{X}\to Y$ the family of abelian varieties induced by the morphism~$\phi$. 
Let us identify the universal covering of~$Y$ with the upper half-plane~$\Po=\Po_1$ and of~$\mathcal{A}_g$ with the Siegel upper half-plane~$\Po_g$, and let~$\varphi\colon\Po\to\Po_g$ denote the map induced by~$\phi$ on the universal coverings. The space~$\Po_g$ is, for every~$g\ge1$, a metric space endowed with the Kobayashi metric~\cite{Kob}. The curve~$\phi\colon Y\to\mathcal{A}_g$ is called \emph{Kobayashi geodesic} if~$\varphi(\Po)\subset\Po_g$ is a totally geodesic subspace for the Kobayashi metric. 
The structure of the variations of Hodge structures~$\V$ associated to a Kobayashi geodesic has been investigated by Möller and Viehweg (\cite{MV}, Theorem 5.1 and Corollary 6.2). 

\begin{theorem*}[Möller-Viehweg]
Let~$Y$ be an affine Kobayashi geodesic coming from a weight-one polarized variation of Hodge structures~$\V$. Assume that~$\V$ is irreducible over~$\Q$. There exists a totally real number field~$F$ of degree~$g$ over~$\Q$ with real embeddings~$\sigma_1,\dots,\sigma_g$ and a rank two variation of Hodge structures~$\Lo_F$ of weight one such that
\[
\V_\C\:=\V\otimes_\Q\C\=\Lo^{\sigma_1}\oplus\cdots\oplus\Lo^{\sigma_g}\,,\quad\text{where}\quad\Lo:=\Lo\otimes_F\C\,.
\]
In particular, the general fiber of~$\mathcal{X}\to Y$ is a polarized abelian varieties with real multiplication by~$F$. 

Moreover, $\phi\colon Y\to \mathcal{A}_g$ can be defined over a number field~$K$.
\end{theorem*}

This result implies that the morphism~$\phi\colon Y\to\mathcal{A}_g$ factors through a Hilbert modular variety (see Section~\ref{sec:HMV}). In particular, the uniformizing group~$\Gamma$ of the base curve~$Y$ can be realized as a subgroup of a Hilbert modular group; it follows that it is defined over a totally real number field of degree~$g$ over~$\Q$. 

\subsubsection{Picard-Fuchs differential equations}
\label{sec:PF}
Let~$\mathbb{L}_j:=\mathbb{L}^{\sigma_j}$, for~$j=1,\dots,g$, be the rank-two variation of Hodge structures appearing in Theorem~\ref{sec:KG}. It follows from the theorem, that each vector bundle~$\Lo_j\otimes\mathcal{O}_{Y}$ is equipped with a Gauss-Manin connection~$\nabla_j=\nabla$, coming from the (Gauss-Manin) connection on~$\V$. The vector bundle~$\Lo_j\otimes\mathcal{O}_{Y}$ contains a holomorphic sub-bundle~$\mathcal{L}_j:=(\Lo_j\otimes\mathcal{O}_{Y})^{(0,1)}$ of rank one, whose global sections can be identified with holomorphic one-forms~$\omega_j$ on the fibers of~$\mathcal{X}\to Y$ that are eigenforms for the action of real multiplication. The bundles~$\mathcal{L}_j$ extend over a compactification~$Y$ of~$Y$ constructed by adding to it a finite number of cusps (see II.5 of~\cite{Deligne}).

Let~$t$ be a coordinate on~$Y$ and let~$\partial/\partial t$ act on the sections of~$\Lo_j\otimes\mathcal{O}_{Y}$ via the connection~$\nabla$, i.e., let $D:=\nabla(\partial/\partial t)\in \mathrm{End}(\Lo_j\otimes\mathcal{O}_{Y})$.
It follows from the irreducibility of~$\mathbb{L}_j$ that the sections~$\omega_j$ and~$D\omega_j$ of~$\Lo_j\otimes\mathcal{O}_{Y}$ are linearly independent, and then~$\omega_j$ satisfies a linear differential equation~$L_j\omega_j=0$ where
\[
L_j\;=\;D^2\+q_jD\+r_j\,,\quad j=1,\dots,g\,,
\]
for some meromorphic functions~$q_j$ and~$r_j$ on a compactification~$\overline{Y}$ of~$Y$. The differential operators~$L_1,\dots,L_g$ are classically called~\emph{Picard-Fuchs differential operators}. If~$Y$ is a genus zero curve, the action of~$D$ induces Fuchsian differential equations on~$\pro^1(\C)$ of the form
\begin{equation}
\label{eq:PF0}
L_j\;=\;p_j(t)\frac{d^2}{dt^2}\+q_j(t)\frac{d}{dt}\+r_j(t)\,,\quad j=1,\dots,g\,,
\end{equation}
where~$p_j(t),q_j(t)$ and~$r_j(t)$ are polynomials, and the degrees of~$q_j(t)$ and~$r_j(t)$ satisfy certain bounds (Fuchs's Theorem) depending on the number of singularities of the differential equation, i.e., on the degree of~$p_j(t)$. 
When~$Y\simeq\Po/\Delta(n,m,\infty)$ is a triangle curve, the Picard-Fuchs operators are hypergeometric differential operators~\cite{BMt},\cite{BN}, but in general non-classical Fuchsian equations arise; for instance, the Picard-Fuchs equations related to Teichmüller curves in~$\mathcal{M}_2$, have been explicitly computed in~\cite{BM}.

\subsubsection{Period mapping and uniformization}
\label{sec:unif}
Kobayashi geodesics in $\A_g$ can be characterized in terms of period maps (\cite{MV}, Theorem 1.2). 
\begin{theorem*}[Möller-Viehweg]
A curve~$\phi\colon Y\to\A_g$ is a Kobayashi geodesic if and only if the associated weight-one variation of Hodge structures~$\V$ over~$Y$ contains a non-unitary irreducible subvariation $\Lo\otimes\mathbb{U}$ of Hodge structures, where $\Lo$ is a rank-two, weight-one variation of Hodge structures on~$Y$ and $\mathbb{U}$ is an unitary irreducible local system, and such that the period map of~$\Lo$ is an isomorphism~$\varphi_1\colon\widetilde{Y}\overset{\simeq}{\to}\Po_1$, where~$\widetilde{Y}$ is the universal covering of~$Y$ and~$\Po_1=\Po$ the period domain. 
\end{theorem*}

Let us interpret this result in terms of differential equations. The variation of Hodge structures~$\Lo$ in the theorem above and~$\Lo=\Lo_1$ in Theorem~\ref{sec:KG} are the same. It follows from the classical theory that the ratio of independent solutions of the Picard-Fuchs differential equation~$L_1$ lifts, as map on the universal covering~$\widetilde{Y}$ of~$Y$, to the period map~$\varphi_1$. What is exceptional in the case of Kobayashi curves, is that this map gives an identification of~$\widetilde{Y}$ with the upper half-plane~$\Po$. This is precisely the property defining another, a priori unrelated, classical object, the~\emph{uniformizing differential equation} of the curve~$Y$ seen as a Riemann surface. More precisely, as envisioned by Poincaré (see~\cite{Gray} for an historical account), the uniformizing differential equation of a hyperbolic Riemann surface~$X$ is the unique linear second-order differential equation with holomorphic coefficients in~$X$, whose ratio of independent solutions identify the universal covering~$\widetilde{X}$ of~$X$ with the simply connected Riemann surface~$\Po$. Two differential equations of order two on~$X$ are called \emph{projectively equivalent} if the ratio of independent solutions lifts to the same function on~$\widetilde{X}$. The above discussion implies then the following result.

\begin{corollary}
\label{cor:uni}
Let~$\phi\colon Y\to\A_g$ be a Kobayashi geodesic. The Picard-Fuchs differential equation associated to the factor~$\Lo$ is projectively equivalent to the uniformizing differential equation of~$Y$ seen as a Riemann surface. In particular, the image of the monodromy representation is conjugated to the uniformizing group~$\Gamma$ of~$Y$ seen as a subgroup of~$\SL_2(\R)$ via the trivial real embedding~$\sigma_1\colon F\to\R$. 
\end{corollary}
 
\subsubsection{Differential equations and modular embeddings} 
\label{sec:modde}
A Theorem of Möller (see~\cite{Mo1}, Corollary 2.11) permits to generalize the content of Corollary~\ref{cor:uni} to the differential operators~$L_j$ for~$j>1$. 
\begin{theorem*}[Möller]
\label{thm:unif}
Let~$\phi\colon Y\to\mathcal{A}_g$ be a Kobayashi geodesic. The ratio of linearly independent solutions of the Picard-Fuchs differential equation associated to~$\Lo_j$, for~$j=1,\dots,g$, lifts, via analytic continuation, to the~$j$-th component~$\varphi_j:=\mathrm{proj}_j\circ\varphi\colon\Po\to\Po$ of the modular embedding~$\varphi\colon\Po\to\Po^g$ associated to~$\phi$.
Moreover, the image of the monodromy representation associated to~$\Lo_j$ is conjugated to~$\sigma_j(\Gamma)\subset\SL_2(\R)$, where~$\sigma_1,\dots,\sigma_g\colon F\to R$ are the real embedding of the totally real number field~$F$. 
\end{theorem*}
The proof in Möller's paper is written in the particular case of Teichmüller curves but, being based on the study of the period map associated to a variations of Hodge structures that decomposes as in Theorem~\ref{sec:KG}, it adapts immediately to the case of Kobayashi curves. 

A first consequence (Corollary 2.11 in~\cite{Mo1}) of Möller's result is that the component~$\varphi_j\colon\Po\to\Po$ of the modular embedding is equivariant for the action of~$\sigma_1(\Gamma)=\Gamma$ on the domain and of~$\sigma_j(\Gamma)$ on the codomain, both acting on~$\Po$ via Möbius transformations. In other words, for every~$\gamma\in\Gamma$ it holds
\begin{equation}
\label{eq:modem}
\varphi_j(\gamma\tau)\=\sigma_j(\gamma)\cdot\varphi_j(\tau)\,\quad\tau\in\Po\,.
\end{equation}

Another consequence of Möller's theorem is the following. Assume that the differential operators~$L_1,\dots,L_g$ are defined over~$\pro^1(\C)$ , i.e., that~$Y$ has genus zero. For every~$j=1,\dots,g$, let~$y_j(t)$ be a holomorphic solution of the Picard-Fuchs equation~$L_jv=0$ around the point~$t=0$. Here the parameter~$t$ on~$\overline{Y}$ should be thought as the image of a modular function (Hauptmodul) giving the identification~$\overline{Y}\simeq\pro^1(\C)$. Via analytic continuation, one constructs a lift~$f_{j}\colon \Po\to\C$  of~$y_j(t)$ on the universal covering of~$Y$, that is holomorphic on~$\Po$. Because of the construction of the Picard-Fuchs differential equations, the functions~$f_j$ are nothing but the period functions~$\int_{\alpha_j}\omega_j(\tau)$ for a suitable choice of cycles~$\alpha_j$ (see~\cite{Mo1} for more on this perspective).
\begin{proposition}
\label{prop:twide}
The functions~$f_j(\tau)$ satisfy the generalized modular property in weight one
\[
f_j(\gamma\tau)\=\bigl(\sigma_j(c)\varphi_j(\tau)+\sigma_j(d)\bigr)\cdot f_j(\tau)\,,\quad\text{for every }\gamma=\begin{pmatrix} a & b\\c & d\end{pmatrix}\in\Gamma\,.
\]
\end{proposition}
\begin{proof}
Let~$\widetilde{y_j}$ be another solution in~$t=0$ of the differential equation~$L_jv=0$, such that~$y_j$ and~$\widetilde{y}_j$ are linearly independent. The result of Möller implies that~$\widetilde{y_j}$ lifts to the complex-valued function~$\varphi_j(\tau)f_j(\tau)$ on~$\Po.$ Moreover, the image of the monodromy representation associated to~$L_jv=0$ is~$\sigma_j(\Gamma)$, where~$\sigma_j$ acts componentwise. Since~$f_j$ and~$\varphi_jf_j$ are lifts of independent solutions of this differential equations, it follows that for every~$\gamma\in\Gamma$ as in the statement, the monodromy action is expressed as
\[
\begin{pmatrix}
\varphi_j(\gamma\tau)f_j(\gamma\tau)\\f_j(\gamma\tau)
\end{pmatrix}
\=\sigma_j(\gamma)\begin{pmatrix}
\varphi_j(\tau)f_j(\tau)\\f_j(\tau)
\end{pmatrix}\=
\begin{pmatrix}
\sigma_j(a)\varphi_j(\tau)f_j(\tau)\+\sigma_j(b)f_j(\tau)\\\sigma_j(c)\varphi_j(\tau)f_j(\tau)\+\sigma_j(d)f_j(\tau)
\end{pmatrix}\,.
\]
\end{proof}

Following~\cite{MZ}, we say that a holomorphic function~$f\colon\Po\to\C$ is a~\emph{twisted modular form} of weight~$(k_1,\dots,k_g)$ with respect to the modular embedding~$\varphi\colon\Po\to\Po^g$ if it satisfies the modular property
\[
f(\gamma\tau)\=f(\tau)\cdot\prod_{j=1}^g(\sigma_j(c)\varphi_j(\tau)+\sigma_j(d))^{k_j}\,,\quad\text{for every }\gamma=\begin{pmatrix} a & b\\c & d\end{pmatrix}\in\Gamma\,,
\]
and the following growth condition: the function~$f(g\tau)\prod_{j=1}^g(\sigma_j(c)\varphi_j(\tau)+\sigma_j(d))^{-k_j}$ is bounded as~$\tau\to\infty$ for every~$g\in\SL_2(K)$. We denote by~$M_{\vec{k}}(\Gamma,\varphi)$ the space of twisted modular forms of weight~$\vec{k}=(k_1,\dots,k_g)$. 
Classical modular forms are twisted modular forms of weight~$(k,0,\dots,0)$. Prominent examples of new twisted modular forms of weight~$(2,0,\dots,0,-2,0,\dots,0)$, with the non-zero weights at the first and~$j$-th place, are the derivative~$\varphi_j'$ of the components of the modular embedding~$\varphi_j$: the modular transformation follows directly from the property~\eqref{eq:modem} of~$\varphi_j$, while the boundedness at the cusp is proven in Proposition 1.1 of~\cite{MZ}.

\begin{remark}
By considering~\emph{stable} forms~$\omega_j$ as in Proposition 5.6 of~\cite{MZ}, it is possible to prove that the functions~$f_j$ in Proposition~\ref{prop:twide} are holomorphic at the cusps in the sense of the above definition, but this is not needed in the arguments below.
\end{remark}

The modularity property of the solutions of the differential equations~$L_jv=0$ gives insight on the local properties of the solutions. In order to describe them, let~$L$ denote a Fuchsian differential equation of order two. The local behavior of the solutions of~$L$ a regular singular point~$t_0$ is described by a pair of numbers~$\bigl(s^{(t_0)}_1,s^{(t_0)}_2\bigr)\in\C^2$ called~\emph{characteristic exponents}. We write~$(s_1,s_2)$ if there is no ambiguity about the singular point considered. 
The difference~$s_2-s_1$ of the characteristic exponents plays a role in the local study of solutions; the relevant cases in our setting are the following (assume~$t_0=0$ for simplicity). If~$s_2-s_1\not\in\Z$, then there are two independent solutions~$y_1(t)$ and~$y_2(t)$ for~$Lv=0$, of the form~$y_i(t)=t^{s_i}\sum_{n\ge 0}{y_{i,n}t^n}$ where the coefficients~$y_{i,n}$ are uniquely determined by a linear recursion, one for each~$i=1,2$, and the initial datum~$y_{i,0}=1$. If~$s_2-s_1=0$ the two solutions are of the form~$y_1(t)=x^{s_1}\sum_{n\ge 0}{y_{1,n}t^n}$ and~$y_2(t)=\log(t)\cdot y_1(t)+\widetilde{y_2}(t)$, where the coefficients~$y_{1,n}$ are uniquely determined recursively with initial datum~$y_{1,0}=1$, and~$\widetilde{y_2}(0)$ is a holomorphic function in~$t=0$ (Frobenius method). For details and the discussion of the missing cases, which do not play a role in this paper, see Chapter 2 of~\cite{Yo}. The table of all the characteristic exponents of the Fuchsian differential operator~$L$ is called the~\emph{Riemann scheme} of~$L$.

In the next lemma, we determine the Riemann scheme of a differential operator~$L_j$ whose solution is a twisted modular form~$f_j$ of weight~$(0,\dots,0,1,0\dots,0)$, non-zero in the~$j$-th place. Examples of such differential operators are the Picard-Fuchs operators associated to a Kobayashi curve in Theorem~\ref{thm:unif}.

\begin{lemma}
\label{lem:Rs}
For every~$j=1,\dots,g$ let~$\{t_{j,1},\dots,t_{j,m_j},\infty\}$ be the set of regular singular points of the differential operator~$L_j$, and assume that~$t_{j,i}=t(c_i)$ for~$i=1,\dots, s$ is the image of a cusp of~$\Gamma$ and that~$t_{j,i+s}=t(\tau_i)$ for~$i=1,\dots,m_j-s$, is the image of a point~$\tau_i\in\Po$ with stabilizer in~$\Gamma$ (or in~$\Gamma/\{\pm1\}$ if~$-1\in\Gamma$) of order~$n_i\ge1$.  
Then the operator~$L_j$ can be normalized to have Riemann scheme of the form
\[
\begin{pmatrix}
t_{j,1} & \cdots & t_{j,s} & t_{j,s+1} &  \cdots & t_{m_j} & \infty\\
0 & \cdots & 0 & 0 & \cdots  & 0 &  s^\infty_{j,1}\\
0 & \cdots & 0 & \frac{1+\mathrm{ord}_{\tau_1}\varphi_j'}{n_1} & \cdots  & \frac{1+\mathrm{ord}_{\tau_{m_j}}\varphi_j'}{n_{m_j}} & s^\infty_{j,2}
\end{pmatrix}
\] 
\end{lemma}
\begin{proof}
We consider singularities given by cusps and by points in~$\Po$ separately. Let~$c$ be a cusp and~$t(c)$ a regular singular point of~$L_j$. Let~$y_{j}(t-t(c))=(t-t(c))^{s_1}\sum_{n\ge0}{y_{j,n}(t-t(c))^n}$ be a holomorphic solutions in~$t(c)$. By assumption, it lifts to a twisted modular form~$f_j$, which has Fourier expansion~$f_j(\tau)=\sum_{m\ge0}f_{j,m}q_c^m$ at the cusp~$c$. Since~$t$ is a Hauptmodul, $t-t(c)$ is a local parameter at the cusp~$c$, and one can express~$q_c$ as a power series in~$t-t(c)$, i.e., $q_c=a_1(t-t()c)+O((t-t(c))^2)$ with~$a_1\neq0$. It follows immediately that~$s_1=\ord_c(f_j)$.
Consider now the second (independent) solution of~$L_jv=0$, which lifts to~$\varphi_jf_j$. It is proven in~\cite{MZ}, Proposition 1.1, that at the cusp~$c$ the modular embedding function has expansion~$\varphi_j(\tau)=a_0\tau\+\sum_{m=0}^\infty{\beta_mq_c^m}$ for some~$a_0\neq0$. Since~$\tau=\log(q_c)$ up to a constant, using again the expression of~$q_c$ in~$(t-t(c))$ we see that the second solution is of the form~$\varphi_j(q_c)f_j(q_c)=\log((t-t(c))y_{j}(t-t(c))+\widetilde{y}_{j}(t-t(c))$ for some holomorphic function~$\widetilde{y}_{j}(t)$. This implies that the characteristic exponents~$(s_1,s_2)$ at the singular point~$t(c)$ satisfy~$s_2=s_1=\ord_c(f_j)$.  

Let now~$\tau_0\in\Po$ be such that~$t(\tau_0)$ is a regular singular point of~$L_j$, and denote by~$n_0$ the order of the stabilizer of~$\tau_0$ in~$\Gamma$ (or $\Gamma/\{\pm 1\}$ if~$-1\in\Gamma$). The local expansion of~$t$ near~$\tau_0$ is of the form
\[
t\=t(\tau_0)\+\sum_{m\ge 0}{t_m(\tau-\tau_0)^{n_0m}}\,,
\] 
with~$t_1\neq0$ since~$t$ is a Hauptmodul. Let~$y_{j}(t-t(\tau_0))=(t-t(\tau_0))^{s_1}\sum_{n\ge0}{y_{j,n}(t-t(\tau_0))^n}$ be the holomorphic solution that lifts to~$f_j$.
One sees immediately, by considering the local expansion at~$\tau_0$ of~$f_j$ as above, that~$s=\ord_{\tau_0}(f_j)/n_0$. The local expansion of~$\varphi_j(\tau)-\tau_0$ is a power series in~$(\tau-\tau_0)^{n_0}$, which implies that the second solutions of~$L_jv=0$ has no logarithmic terms, i.e., the singular point~$t(\tau_0)$ has characteristic exponents~$s_1\neq s_2$. This implies that the second solution is of the form~$y_{j,2}(t-t(\tau_0))=(t-t(\tau_0))^{s_2}\widetilde{y}_{j,2}(t-t(\tau_0))$ for some holomorphic function~$\widetilde{y}_{j,2}$, non zero in~$t(\tau_0)$.  
Again by comparing the local expansion of~$\varphi_j-\varphi_j(\tau_0)=y_{j,2}(t-t(\tau_0))/y_{j,1}(t-t(\tau_0))=(t(\tau)-t(\tau_0))^{s_2-s_1}(1+\cdots)$, one finds that~$s_2-s_1=\ord_{\tau_0}(\varphi_j-\varphi_j(\tau_0))/n_0=(\ord_{\tau_0}\varphi_j'+1)/n_0$.

The Riemann scheme of the differential equation~$L_j$ is then
\[
\begin{pmatrix}
t(c_1) & \cdots & t(c_s) & t(\tau_{1}) & \cdots & t(\tau_{m_j}) & \infty\\
\ord_{c_1}(f_j) & \cdots & \ord_{c_s}(f_j) & \ord_{\tau_{1}}(f_j)/n_1 & \cdots  & \ord_{\tau_{m_j}}(f_j)/n_{m_j} & \ord_{\tau_\infty}(f_j)/n_{\infty}\\
\ord_{c_1}(f_j) & \cdots & \ord_{c_s}(f_j) & \frac{(\ord_{\tau_1}(f_j)+\ord_{\tau_1}\varphi_j'+1)}{n_1} & \cdots & \frac{(\ord_{\tau_{m_j}}(f_j)+\ord_{\tau_{m_j}}\varphi_j'+1)}{n_{m_j}} & \frac{(\ord_{\tau_{\infty}}(f_j)+\ord_{\tau_{\infty}}\varphi_j'+1)}{n_\infty} 
\end{pmatrix}\,,
\]
where~$n_\infty=1$ if~$\tau_\infty$, the point where~$t$ has a pole, is either in~$\Po$ with trivial stabilizer or a cusp. In the latter case it also holds~$\ord_{\tau_{\infty}}\varphi_j'=0$.

The transformation~$v\mapsto \widetilde{v}=\prod_{i=1}^{m_j}(t_i-t)^{-s_1^{(j)}}\cdot v$ of the dependent variable~$v$ in~$L_jv=0$ produces a solution of an equivalent differential equation~$\widetilde{L_j}\widetilde{v}=0$ whose first characteristic exponent is~$0$ at all finite singular points, and whose second characteristic exponent at~$t_i$ is the difference~$s_2^{(t_i)}-s_1^{(t_i)}$ of the characteristic exponents of~$t_i$ with respect to the differential operator~$L_j$, as follows from an easy calculation. Then~$\widetilde{L_j}$ is the desired normalization of~$L_j$.
\end{proof}

\begin{remark}
The proof shows, by studying the differential equation at~$\infty$ with the parameter~$1/t$, that the holomorphic solution lifts to a modular forms with divisor supported at the unique point where the Hauptmodule~$t$ has a pole. In particular,~$\eta_1\neq0$,
\end{remark}

\begin{lemma}
\label{lem:rec}
Let the differential operator~$L_j$ be normalized as in Lemma~\ref{lem:Rs}. 
Assume that~$t_1=0$ is a regular singular point with characteristic exponents~$(0,0)$, and let~$y_j(t)=1+\sum_{n\ge1}y_{j,n}t^n$ be the unique normalized holomorphic solution in~$t=0$. The coefficients~$\{y_{j,n}\}_n$ of~$y_j(t)$ satisfy a recursion
\[
d_{j,0}(n)\cdot y_{j,n+1}\=d_{j,1}(n)\cdot y_{j,n}+\cdots d_{j,m_j-1}(n)\cdot y_{j,n-m_j+1}\,,
\]
for every~$n\ge1$ with inital datum~$y_{j,0}=1$, where~$d_{j,0},\dots,d_{j,m_j-1}$ are polynomials in~$n$ which depends on the coefficients of~$L_j$. Moreover
\[
d_{j,0}\=(n+1)^2\prod_{i=2}^{m_j}(-t_{j,i})\,,\quad
d_{j,m_j-1}(n)\=(n-m_j+2+s^\infty_{j,1})(n-m_j+2+s^\infty_{j,2})\,.
\]
\end{lemma}
\begin{proof}

That the coefficients of a holomorphic solution of a Fuchsian equation satisfy a linear recursion of length~$m_j$ is well known (Frobenius method). We only have to compute the coefficients~$d_0$ and~$d_{m_j-1}$. Since the top row of the Riemann scheme of~$L_j$ has zeros at all the finite singularities, it follows that~$L_j$ is of the form~$L_j=\frac{d^2}{dt^2}+\hat{q}_j(t)\frac{d}{dt}+\hat{r}_j(t)$, where
\begin{equation}
\label{eq:p,q,r}
\hat{q}_j(t)=\sum_{i=1}^{m_j}\frac{a_{j,i}}{t-t_{j,i}}\=\frac{q_j(t)}{tp_j(t)}\,,\quad \hat{r}_j(t)=\sum_{i=1}^{m_j}\frac{c_{j,i}}{t-t_{j,i}}\=\frac{r_j(t)}{tp_j(t)}\,,\quad a_{j,i},c_{j,i}\in\C
\end{equation}
where~$1-a_{j,i}=s_2^{(t_{j,i})}+s_1^{(t_{j,i})}$ and~$p_j(t)=\prod_{i=2}^{m_j}(t-t_{j,i})$ (see~\cite{Yo}, Chapter 2.6). 
The coefficients~$a_{j,i}$ and~$c_{j,i}$ are related with the characteristic exponents~$s_{j,1}^\infty$ and~$s_{j,2}^\infty$ at~$\infty$ via 
\begin{equation}
\label{eq:indi}
(x-s^\infty_{j,1})(x-s^\infty_{j,2})\= x(x-1)+\biggl(2-\sum_{i=1}^{m_j}{a_{j,i}}\biggr)x+\sum_{i=1}^{m_j}{c_{j,i}t_{j,i}}
\end{equation}
as follows from the computation of the the indicial equation of~$L$ at~$\infty$. Frobenius's method implies that the coefficient~$d_{j,m_j-1}$ of the recursion is given by
\[
d_{j,m_j-1}(n)\=(n-m_j+2)(n-m_j+1)\cdot p_{j,m_j-1}\+(n-m_j+2)\cdot q_{j,m_j-1}+r_{j,m_j-2}\,,
\]
where~$p_{j,a},q_{j,a}$, and~$r_{j,a}$ denote the coefficient of~$t^a$ of~$p_j(t),q_j(t),$ and~$r_j(t)$ respectively. From the definition of these polynomials~\eqref{eq:p,q,r}, it follows that
\[
d_{j,m_j-1}(n)\=(n-m_j+2)(n-m_j+1)\+(n-m_j+2)\sum_{i=1}^{m_j}{a_{j,i}}+\sum_{i=1}^{m_j}{c_{j,i}t_{j,i}}\,,
\]
which, thanks to the relation~\eqref{eq:indi}, equals the expression for~$d_{j,m_j-1}$ in the statement. 

Similarly, the coefficient~$d_{j,0}$ of the recursion is identified via Frobenius's method with 
\[
d_{j,0}(n)\=n(n+1)\cdot p_{j,0}\+(n+1)\cdot q_{j,0}\,.
\]
The definition of~$p_j(t)$ and~$q_j(t)$ implies immediately that~$p_{j,0}=q_{j,0}=\prod_{i=2}^{m_j}(-t_{j,i})$.
\end{proof}

\subsection{Hilbert modular varieties}
\label{sec:HMV}
Let~$F$ be a totally real field of degree~$g$ over~$\Q$ and let~$\OO_F$ denote its ring of integers. Let~$R$ be a ring. An~\emph{abelian variety with real multiplication by $\mathcal{O}_F$} is an abelian variety~$A\to\mathrm{Spec}(R)$ of dimension~$g$ together with a ring injection~$\iota\colon\mathcal{O}_F\to\mathrm{End}_R(A)$ such that the following condition holds. Let~$A^\vee$ be the dual abelian variety, and notice that~$\mathcal{O}_F$ acts on~$A^\vee$ by duality. Define 
\[
M_A:=\{\lambda\colon A\to A^t\,:\,\lambda\circ\iota(r)=\iota(r)^\vee\circ\lambda\,,\,\forall r\in\OO_F \}
\]
to be the~$\OO_F$-module of symmetric~$\OO_F$-linear homomorphisms from $A$ to $A^\vee$. We assume that the Deligne-Pappas condition~$A\otimes_{\OO_F}M_A\simeq A^\vee$ holds (this condition is always satisfied if~$\mathrm{char}(R)=0$). See~\cite{GoBook} for details. 

For~$n\ge3$ there exists a fine moduli scheme~$\mathcal{M}_F^n\to\Spec(\Z[(nd_F)^{-1}])$, where~$d_F$ is the discriminant of~$F$, classifying isomorphism classes of abelian varieties with real multiplication by~$\OO_F$ and full level~$n$ structure. By forgetting the level structure one gets a scheme~$\mathcal{M}_F\to\Spec(\Z)$, which is a coarse moduli space for abelian varieties with real multiplication by~$\OO_F$. 
The scheme~$\mathcal{M}_F$ is the disjoint union of components~$\mathcal{M}_F(\fa)$, where~$\fa$ runs over the the strict class group~$Cl(F)^+$ of~$F$. For~$R=\C$ the component~$\mathcal{M}_L(\fa)(\C)$ corresponds to the quotient of~$\Po^g$ by the Hilbert modular group~$\SL_2(\OO_F\oplus\fa)$. 

\smallskip

Let~$p\in\Z$ be a prime unramified in~$F$. Then
\[
pF=\prod_{l=1}^r\fp_l\,,\quad f_l:=\deg[\mathcal{O}_F/\fp_l:\F_p]\,,
\]
where~$\fp_1,\dots,\fp_r$ are prime ideals in~$F$. Denote by~$\F$ the minimal finite field that contains copies of every residue field~$k_l:=\mathcal{O}_F/\fp_l$ (one has $[\F:\F_p]=\mathrm{lcm}(f_1,\dots,f_r)$), and denote by~$W(\F)$ its infinite Witt vectors. There is an idempotent decomposition 
\begin{equation}
\label{eq:idemp}
\mathcal{O}_F\otimes_\Z W(\F)\=\bigoplus_{\mathrm{Emb(\OO_F,W(\F))}}W(\F)
\end{equation}
with idempotents~$e_{(l,i)}$ for~$1\le l\le r$ and~$1\le i\le f_l$. We write accordingly~$\sigma_{(l,i)}$ for the embeddings of~$\OO_F$ into $W(\F)$ and assume that the indexing is chosen so that
\begin{equation}
\label{eq:normf}
\sigma\circ\sigma_{(l,i)}\=\sigma_{(l,i+1)}\quad \forall\,l,i\,
\end{equation}
where~$\sigma$ is the Frobenius homomorphism on~$W(\F)$, with the convention that~$(j,f_j+1)=(j,1)$. We still use~$e_{(l,i)}$ and~$\sigma_{(l,i)}$ to denote the corresponding objects modulo~$p$.

\smallskip

Let~$R_0$ be an~$\F$-algebra and consider an Abelian variety~$g\colon A\to\Spec(R_0)$ with real multiplication by~$\OO_F$. The space~$g_*\Omega^1_{A/R}$ is a free~$\OO_F\otimes R_0$-module of rank one, and let~$\omega\in g_*\Omega^1_{A/R_0}$ be a generator (see Section 2 of \cite{GoHasse} for details). The elements
\begin{equation}
\label{eq:basisRM}
\omega_{l,i}:=e_{(l,i)}\omega\,,\quad 1\le l\le r\,,\;1\le i\le f_l
\end{equation}
form an~$R_0$-basis of~$g_*\Omega^1_{A/R}$ which respects the action of real multiplication, i.e., each~$\omega_{(l,i)}$ is an eigenform for the action of~$\OO_F$ on~$g_*\Omega^1_{A/R_0}$. Thanks to the normalization~\eqref{eq:normf}, the action of the Cartier operator~$\mathcal{C}$ on~$H^0(A, \Omega^1_{A/R_0})$ is given for every pair of indices~$(l,i)$ simply by~$\mathcal{C}(\omega_{(l,i)})=\alpha_{l,i}\cdot\omega_{(l,i+1)}$ for some~$\alpha_{l,i}\in R_0$.
In other words, the action of~$\mathcal{C}$ with respect to the basis~$\{\omega_{(l,i)}\}_{(l,i)}$ is represented by a block matrix with non-zero square blocks of size~$f_l$ for~$1\le l\le r$ and zero elsewhere.

\subsection{Curves in Hilbert modular varieties}
\label{sec:CHMV}
Let $\phi\colon Y\to\mathcal{A}_g$ be an affine Kobayashi geodesic, seen as a curve in a (complex) Hilbert modular variety~$\mathcal{M}_F(\C)$ (Theorem~\ref{sec:KG}). By a slight abuse of notation, we denote it by~$\phi\colon Y\to \mathcal{M}_F$.
If~$\mathcal{M}_F^*$ denotes a toroidal compactification of the Hilbert modular scheme~\cite{Chai}, we denote by~$\overline{\phi}\colon\overline{Y}\to\mathcal{M}_F^*$ the corresponding compactificaation of~$\phi\colon Y\to \mathcal{M}_F$ to~$Y$, and define~$C:=\overline{Y}\setminus Y$ to be the finite set of cusps. 

Since~$\phi\colon Y\to\A_g$ can be defined over a number field~$K$ and~$\mathcal{M}_F$ is defined over $\Spec(\Z)$, by considering the Zarisky closure of~$Y$ along the embedding~$\phi\colon Y\to\mathcal{M}_F$, the Kobayashi geodesic~$\phi\colon Y\to\A_g$ can be endowed with an integral structure. From now on we consider~$\phi\colon Y\to\A_g$ as given by a model~$f\colon \mathcal{X}\to Y$ defined over~$\Spec(R)$ for some ring~$R=\mathcal{O}_K[S^{-1}]$, where~$\mathcal{O}_K$ is the ring of integer of the number field~$K$, and~$S\subset\Z$ is a finite set of primes. Similarly, we denote by ~$\overline{f}\colon\overline{\mathcal{X}}\to\overline{Y}$ a model for the compactification~$\overline{Y}\to\mathcal{M}_F^*$ defined over~$\Spec(R)$.
At the expense of enlarging the set of primes~$S$, we assume that the model~$\overline{f}\colon \overline{\mathcal{X}}\to \overline{Y}$ has good reduction at every prime~$\fp\in R$ above the rational primes outside~$S$. We are currently not able to determine the set~$S$ for a general Kobayashi curve (it has been determined for certain models of Teichmüller curves in the moduli space of genus two curves and discriminant~$D=13,17$ in Section 10 of~\cite{BM}); nevertheless, ~$S$ must contain the primes that ramify in~$F$ and in~$K$, and the primes~$p$ for which the solution space of the differential equation~$L_i=0\mod p$ has dimension bigger than one (over~$\F(t^p)$). This, as follows from a result of Katz~\cite{K70}, only happens for finitely many primes. Later, we will consider a possibly larger set~$S$ containing also the primes dividing the order of the elliptic points of~$Y$.

Let~$B$ be an~$R$ algebra. A consequence of Theorem~\ref{sec:KG} is that the set~$\{\omega_i\}_{i=1}^g$ (notation as in Section~\ref{sec:PF}) gives a~$B$-basis of~$f_*\Omega^1_{\mathcal{X}/B}$ made of holomorphic eigenforms for the action of real multiplication on the fibers of~$f\colon \mathcal{X}\to Y$. The base change~$\mathcal{X}_p:=\mathcal{X}\otimes_B\bigl(B\otimes_\Z\F\bigr)$ makes~$\mathcal{X}\to\Spec(B)$ into a scheme~$f_p\colon\mathcal{X}_p\to R_0$ over the~$\F$-algebra~$R_0:=B\otimes_\Z\F$. The~$R_0$-module~$(f_p)_*\Omega^1_{\mathcal{X}_p/R_0}$ has an~$R_0$-basis~$\{\omega_{0j}\}_{j=1}^g$ obtained by the base-changing the basis~$\{\omega_j\}_{j=1}^g$ coming from the action of real multiplication on~$\mathcal{X}\to Y$, and it has the basis~$\{\omega_{(l,i)}\}_{(l,i)}$ coming from the idempotent decomposition~\eqref{eq:basisRM}. Since the elements of both basis are eigenforms for the action of~$\OO_L$ and~$p\not\in S$ is of good reduction, one has the following result comparing the two basis. 
\begin{lemma}
\label{lem:basisrel}
Let~$\{\omega_{0j}\}_{j=1}^g$ and~$\{\omega_{(l,i)}\}_{(l,i)}$, for~$l=1,\dots,r$ and~$i=1,\dots,f_i$, be the basis of eigenforms of~$(f_p)_*\Omega^1_{\mathcal{X}_p/R_0}$ described above. For every~$j=1,\dots,g$ there exists a choice of indices~$(l_j,i_j)$ and a non-zero~$\beta_{j}\in R_0$ such that
\[
\omega_{0j}\=\beta_j\cdot\omega_{(l_j,i_j)}\,,
\]
and~$(l_j,i_j)\neq(l_k,i_k)$ if $j\neq k$.
\end{lemma}

\subsection{Expansion coefficients and Cartier matrix}
\label{sec:expc}
The solutions of Picard-Fuchs differential equations have been studied by Katz in terms of expansion coefficients, which are defined below in the case of one-forms. The details and the general result can be found in Katz's paper~\cite{K84}.

Let~$R$ be a Noetherian ring, let~$B$ be a smooth $R$-algebra, and let~$X$ be a smooth~$B$-scheme of relative dimension~$N\ge1$. Let~$O\in X(B)$ be a marked point and let~$(T_1,\dots,T_N)$ be local coordinates on~$X$ at~$O$. 

Let~$\widehat{X}$ denote the formal completion of $X$ along~$O$. 
Via an isomorphism of pointed formal schemes~$\mathrm{Spf}(B[\![T_1,\dots,T_N]\!])\simeq \widehat{X}$, to an element~$\omega\in H^0_{dR}\bigl(X,\Omega^1_{X/B})$ is associated the formal expansion
\begin{equation}
\label{eq:expcoef}
\omega\;\sim\;\sum_{1\le k\le N}\sum_{I\in(\Z_{\ge0})^N}{a(\omega;k,I)\cdot T^I\,\frac{dT_k}{T_k}}\,,\quad a(\omega;k,I)\in B
\end{equation}
where~$T^I:=\prod_{v=1}^N{T_v^{I_v}}$ if~$I=(I_1,\dots,I_N)$. The elements~$a(\omega;k,I)$ are called~\emph{expansion coefficients of $\omega$}. 

Let~$k$ be a perfect field of characteristic~$p>0$ and set~$R=W(k)$ its ring of Witt vectors. Denote by~$B_0:=B\otimes_R{k}$ and $X_0:=X\otimes_B{B_0}$ the corresponding objects modulo~$p$. The reduction of the expansion coefficients describe the action of the Cartier operator~$\mathcal{C}\colon H^0(X_0,\Omega^1_{X_0/B_0})\to H^0(X_0^{(p)},\Omega^1_{X_0^{(p)}/B_0})$ with respect to a suitable basis as follows. Given~$\omega\in\Omega^1_{X_0/B_0}$ with formal expansion~\eqref{eq:expcoef} the action of~$\mathcal{C}$ on~$\omega$ is given by (see Corollary~5.3 of~\cite{K85})
\begin{equation}
\label{eq:Cexp}
\mathcal{C}(\omega)\=\sum_{1\le k\le N}\sum_{I\in (\Z_{\ge0})^N}{a(\omega;k,pI)\cdot T^I\,\frac{dT_k}{T_k}}\,.
\end{equation}
There exists (Lemma 4.1 of~\cite{K85}) a set of indices~$(k_v,E_v)$, where~$v=1,\dots,N$, for which the map
\[
H^0(X,\Omega^1_{X/B})\to R^N\,,\quad \omega\mapsto(,\cdots,a(\omega; k_v,I_v),\dots,)
\]
is an isomorphism of~$R$-modules. Fix such a choice of indices, and let~$\eta=\{\eta_1,\dots,\eta_N\}$ be the basis of~$H^0(X,\Omega^1_{X/B})$, such that~$a(\eta_{i};k_v,W_v)=\delta_{i,v}$, where~$\delta_{i,v}$ is the Kronecker symbol. 
We denote by $\eta_0=\{{\eta_0}_1,\dots,{\eta_0}_N\}$ the corresponding object modulo~$p$.
For~$m\ge0$ define the matrix~$E(\eta;m)$ by
\begin{equation}
\label{eq:Em}
E(\eta;m)_{v,j}\:= a(\eta_j; k_v,p^mI_v)\,,\quad v\,,j\=1,\dots,N\,.
\end{equation}
By using~\eqref{eq:Cexp}, the action of the Cartier operator on~$\{\eta_{01},\dots,\eta_{0N}\}$ can be described in terms of the matrix~$E(m)$ (see Lemma 6.1 in~\cite{K85}).  
\begin{lemma}
\label{lem:Em}
Let~$E(\eta;m)$ be as in~\eqref{eq:Em}. Then $E(\eta;m)\mod p$ is the matrix of the $m$-th iterate of the Cartier operator~$\mathcal{C}$ with respect to the bases~$\{\eta_{01},\dots,\eta_{0N}\}$ and~$\{\eta_{01}^{(p^m)},\dots,\eta_{0N}^{(p^m)}\}$. In particular, the following congruence
\[
E(\eta;m+1)\=E(\eta;m)^{(p)}E(\eta;1)\mod p,
\] 
where~$E(m)^{(p)}$ is obtained by raising to the~$p$-th power every entry of~$E(m)$, holds for every~$m\ge0$.
\end{lemma}

\subsection{Integrality of solutions of differential equations}
Let~$\phi\colon Y\to\A_g$ be a genus zero affine Kobayashi geodesics with parameter~$t$, and let~$f\colon{\mathcal{X}}\to Y$ be the associated integral model defined over the ring~$R=\mathcal{O}_K[S^{-1}]$ as in Section~\ref{sec:CHMV}. In this section we assume, possibly enlarging~$S$, that the primes dividing the order of elliptic points of~$Y$ are in~$S$. 
Fix a basis~$\{\omega_j\}_{j=1}^g$ of~$f_*\Omega^1_{\mathcal{X}/R}$ of eigenforms for the action of real multiplication by~$\OO_L$, corresponding to the holomorphic eigenforms~$\omega_j\in\mathcal{L}_j\subset\mathbb{L}_j$ for~$j=1,\dots,g$, of Section~\ref{sec:PF}. Let~$L_1,\dots,L_g$ be the associated Picard-Fuchs differential operators as in~\eqref{eq:PF0}. 
Since~$Y$ has genus zero and is affine, by removing from it the points corresponding to singularities of the differential operators~$L_j$, we consider the family~$\mathcal{X}\to Y$ as defined over the~$R$-algebra $B:=R[t][\mathrm{lcm}(p_1(t),\dots,p_g(t))^{-1}]$ where~$p_j(t)$ is as in~\eqref{eq:PF0}. 

By construction, the differential operator~$L_j$ satisfies~$\nabla(L_j)\omega_j=0$ in cohomology, where~$\nabla$ is the Gauss-Manin connection and~$\omega_j\in\mathcal{L}_j$. The link between the theory over number fields and in postive characteristic is given via expansion coefficients by the following result of Katz.
\begin{theorem*}[Katz~\cite{K84}]
\label{thm:K84}
Consider the formal expansion attached to~$\omega_j\in\mathcal{L}_j$ 
\[
\omega_j\;\sim\;\sum_{1\le k\le g}\sum_{I\in(\Z_{\ge0})^g}{a(\omega_j;k,I)\cdot T^I\,\frac{dT_k}{T_k}}\,,\quad j=1,\dots,g\,.
\]
If~$\nabla(L_j)(\omega_j)=0$ holds in~$H^1_{dR}(\mathcal{X}/B)$, then
\[
L_j\,a(\omega_j,k,I)\equiv 0\mod I_k B\,,
\]
where~$I=(I_1,\dots I_g)\in(\Z_{\ge0})^g$.
\end{theorem*}

Let~$p\not\in S$ be a rational prime, and let~$\F$ be the finite field associated to~$p$ in Section~\ref{sec:HMV}. Consider the model~$f\colon \mathcal{X}\to Y$ as defined over~$B\otimes W(\F)$. 
As it follows from Lemma~\ref{lem:Em}, certain linear combinations of the expansion coefficients~$a(\omega_j;k,I)$ describe, after reduction modulo~$p$, the action of $m$-th iterated of the Cartier operator with respect to the basis~$\{\omega_{01},\dots,\omega_{0g}\}$ and~$\{\omega_{01}^{(p^m)},\dots,\omega_{0g}^{(p^m)}\}$ obtained by reduction modulo~$p$. We reorder as follows the basis in order to make the shape of the Cartier matrix particularly easy to work with. Recall from Lemma~\ref{lem:basisrel} that the modulo~$p$ reduction~$\omega_{0j}$ of~$\omega_j$ corresponds to a differential~$\omega_{(l,i)}$ for some index of~$(l,i)$ with~$l\in\{1,\dots,r\}$ and~$i\in\{1,\dots,f_l\}$, where~$p\mathcal{O}_F=\prod_{l=1}^r\fp_l$ and~$f_l=[\mathcal{O}_F/\fp_l:\F_p]$.
Up to rename the indices, one can assume that 

\begin{equation}
\label{eq:Nj}
\omega_{j}\overset{\mod p}{\longrightarrow}\omega_{(N_j+1,i_j)}\,,\quad\text{if}\quad j=\sum_{l=0}^{N_j}{f_l}+i_j\,,\quad 0\le N_j\le r-1\,,\quad 1\le i_j\le f_{N_j+1}\,,
\end{equation}
where we set~$f_0:=0$.  
\begin{example}
Let~$F$ be a cubic extension and let~$p\in\Z$ decompose in~$\mathcal{O}_F$ as~$p=\fp_1\cdot\fp_2$ with~$f_1=2$ and~$f_2=1$.
Then~$N_j\in\{0,1\}$ and, if~$N_j=0$, then~$1\le j_0\le f_1=2$; if~$N_j=1$, then~$1\le j_0\le f_2=1$. The possible assignments~$j\to\sum_{l=0}^{N_j}f_l+i_j$ are then
\[
1=f_0+1=0,\,,\quad 2=f_0+2\,,\quad 3=f_1+1\,,
\]
corresponding to~$\omega_1\to\omega_{(1,1)}$ and~$\omega_2\to\omega_{(1,2)}$, and~$\omega_3\to\omega_{(2,1)}$.
\end{example}

Denote by~$\omega=\{\omega_1,\dots,\omega_g\}$ the basis ordered as above, and let~$M^\eta_\omega\in R^{g\times g}$ be the change of basis matrix from $\eta$ to~$\omega$. Define
\begin{equation}
\label{eq:cart}
\hat{E}(\omega;m)\:=\bigl(M^\eta_\omega\bigr)^{(p^m)}\cdot E(\eta;m)\cdot \bigl(M^\eta_\omega\bigr)^{-1}\,.
\end{equation}
It is clear from~Lemma~\ref{lem:Em} that~$\hat{E}(\omega;m)\mod p$ represents the action of the $m$-th iterate of the Cartier operator with respect to the basis~$\omega_0=\{{\omega_0}_1,\dots,{\omega_0}_g\}$ and~$\{{\omega_0}_1^{(p^m)},\dots,{\omega_0}_g^{(p^m)}\}$, and that the collection~$\{\hat{E}(\omega;m)\}_m$ satisfies the same congruences modulo~$p$ as~$\{E(\eta;m)\}_m$ in the same lemma. Note that each matrix-entry~$\hat{E}(\omega;m)_{\nu,j}$ is a finite linear combination of expansion coefficients~$a(\omega_j;k,p^mI)$ depending only on the index~$j$. It follows from Katz's result that $\hat{E}(\omega;m)_{\nu,j}$ is a solution modulo~$p^m$ of the differential equation~$L_jv(t)=0$. It may be the case though, that~$\hat{E}(\omega;m)_{\nu,j}$ is zero modulo~$p^m$. 

\begin{proposition}
\label{prop:cong}
Let~$m\in\Z_{\ge1}$ be fixed. 
For every~$j\in\{1,\dots,g\}$ let~$0\le N_j\le r-1$ and~$1\le i_j\le f_{N_j+1}$ be as in~\eqref{eq:Nj}. For each~$i_j$, let~$\nu_{i_j}$ be the unique integer such that~$1\le\nu_{i_j}\le f_{N_j+1}$ and~$\nu_0\equiv i_j+m\mod f_{N_j+1}$, and set~$\nu_j:=\sum_{l=0}^{N_j}f_l+\nu_{i_j}$. 
Then
\[
\hat{E}(\omega;m)_{\nu_j,j}\neq 0\mod p\cdot B\otimes W(\F)\,,\quad \text{and}\quad L_j\hat{E}(\omega;m)_{\nu_j,j}\equiv0\mod p^m\cdot B\otimes W(\F)\,.
\]
\end{proposition}
\begin{proof}
The second statement follows immediately from Katz's theorem in~\cite{K84} and the discussion before Proposition~\ref{prop:cong}. The claim that~$\hat{E}(\omega;m)_{\nu_j,j}$ is non trivial modulo~$p$ (and in particular, modulo~$p^m$) has to be proven.

The ordering of the basis~$\omega$ in~\eqref{eq:Nj} implies that~$\hat{E}(\omega;m)\mod p$ represents the action of the $m$-th iterated of the Cartier operator with respect to the ordered basis
\[
\{\omega_{(1,1)},\dots,\omega_{(1,f_1)},\omega_{(2,1)},\dots,\omega_{(r,f_r)}\}\,,\quad\text{and}\quad\{\omega^{(p^m)}_{(1,1)},\dots,\omega^{(p^m)}_{(1,f_1)},\omega^{(p^m)}_{(2,1)},\dots,\omega^{(p^m)}_{(r,f_r)}\}\,.
\]
It follows from the definition (see Section~\ref{sec:HMV}) of~$\{\omega_{(l,i)}\}_{(l,i)}$ and the normalization~\eqref{eq:normf} that 
\begin{equation}
\label{eq:shape}
\hat{E}(\omega;1)\;\equiv\; \begin{pmatrix}
M_1 & & &  \\
& M_2 & &  \\
& & \ddots & \\
& & & M_r 
\end{pmatrix}\mod p
\end{equation}
with non-zero blocks~$M_l\in\mathrm{Mat}_{f_l\times f_l}(B\otimes\F)$ for~$l=1,\dots,r$, and zero everywhere else. 
Each block~$M_l$ describes the action of the Cartier operator on the Cartier-stable subspace of differentials~$\{\omega_{(l,i)}\,:\,i=1,\dots,f_l\}$ and is of the form
\[
M_l\=\begin{pmatrix}
0 & \cdots & 0 & \ast\\
\ast & 0 &\cdots & 0 \\
0 & \ddots & \ddots &\vdots \\
0 & 0 & \ast & 0
\end{pmatrix}\,,
\]
where~$\ast$ stands for a non-zero entry.  

From the congruence~$\hat{E}(\omega;m+1)\equiv \hat{E}(\omega;m)^{(p)}\cdot\hat{E}(\omega;1)\mod p$ it follows that $\hat{E}(\omega;m)\mod p$ has the same block structure ~\eqref{eq:shape} for every~$m\ge1$. The non-zero entries of~$\hat{E}(\omega;m) \mod p$ are then the non-zero entries of its sub-blocks~$M_l(m)=M_l(m-1)^{(p)}M_l$ for~$l=1,\dots,r$. A simple computation shows that~$M_l(m)_{\nu_0,j_0}\neq 0$ if and only if the indices are of the form~$(\nu_0, f_l-m+\nu_0)$, or equivalently if~$(\nu_0,j_0)=(j_0+m-f_{l},j_0)$, with the usual convention~$f_l+1=1$. The proposition then follows from the shape~\eqref{eq:shape} of~$\hat{E}(\omega,m)\mod p$ and the size of the blocks~$M_l$.
\end{proof}

Now fix~$j\in\{1,\dots,g\}$ and let~$m\ge 1$ be an integer. Denote by~$y_j^{[m]}(t)$ a solution of~$L_j$ modulo~$p^m$. Such solutions arise from the entries of the matrix~$\widehat{E}(\omega;m)$ as shown in Proposition~\ref{prop:cong}. 

In the case~$y_j^{[m]}(t)$  is a polynomial (this happens for instance if~$Y$ is torsion free), we study its degree in terms of~$m$. 
\begin{lemma}
\label{lem:deg}
Let~$y_j^{[m]}(t)\in R[t]$ be a polynomial such that~$L_j y_j^{[m]}(t)\equiv0\mod p^m$ holds, and denote by~$\overline{y_j^{[m]}}$ its modulo~$p^m$ reduction. There exists an integer~$\kappa_j\ge0$ such that
\[
\deg{\overline{y_j^{[m]}}(t)}\;\ge\;p^{\lceil\frac{m}{2}\rceil}\-\kappa_j\,, 
\]
where~$\lceil{x}\rceil$ denotes the ceiling function. 
In particular, $\deg{\overline{y^{[m]}_j}(t)}\to\infty$ as~$m\to\infty$.
\end{lemma}
\begin{proof}
Write~${y_j^{[m]}(t)}=\sum_{n\ge0}{y_{j,n}^{[m]}t^n}$. Let~$\beta_{j}^{[m]}$ be the smallest integer such that~$y_{j,\beta_j^{[m]}}^{[m]}\neq0$ and~$y_{j,i}^{[m]}=0$ for~$i=\beta_j^{[m]}+1,\dots,\beta_j^{[m]}+m_j$, where~$m_j$ is the number of regular singular points of~$L_j$ (such~$\beta_j^{[m]}$ exists since~$Y_j^{[m]}$ is a polynomial solution). Since~$y_j^{[m]}$ is a solution of~$L_jv\equiv0\mod p^m$, its coefficients satisfy the recursion in Lemma~\ref{lem:rec} modulo~$p^m$. It holds in particular
\[
d_{j,0}(\beta_j^{[m]}+m_j-1)\cdot 0\;\equiv\;d_{j,1}(\beta_j^{[m]}+m_j-1)\cdot 0\+\cdots\+d_{j,m_j-1}(\beta_j^{[m]}+m_j-1)\cdot y_{j,\beta_j^{[m]}}^{[m]}\mod p^m\,.
\]
Since~$y_{j,\beta_j^{[m]}}^{[m]}\neq0$, the above congruence and the description of~$d_{j,m_j-1}$ in Lemma~\ref{lem:rec} imply
\[
0\equiv d_{j,m_j-1}(\beta_j^{[m]}+m_j-1)=(\beta_j^{[m]}+1+s^\infty_{j,1})(\beta_j^{[m]}+1+s^\infty_{j,2})\mod p^m\,,
\]
which proves the lemma with constant~$\kappa_j:=1+\max\{s^\infty_{j,1},s^\infty_{j,2}\}$. 
\end{proof}

\begin{theorem}
\label{thm:main}
Let~$Y\hookrightarrow\mathcal{A}_g$ be an affine Kobayashi geodesic defined with an integral model over~$\mathcal{O}_K[S^{-1}]$. For~$j\in\{1,\dots,g\}$ let~$L_j$ be the Picard-Fuchs operators coming from~$Y\hookrightarrow\mathcal{A}_g$, normalized as in Lemma~\ref{lem:Rs}, and assume that~$t_1=0$ is a regular singular point of~$L_j$ with local exponents~$(0,0)$.  The holomorphic solution~$y_j(t)=1+\sum_{n\ge1}{y_{j,n}t^n}$ of~$L_jv=0$ at~$t=0$ has coefficients in~$R=\mathcal{O}_K[S^{-1}]$.
\end{theorem}
\begin{proof}
Let~$p\not\in S$ be a prime, and consider the collection~$\{y_j^{[m]}(t)\}_{m\ge1}$ of solutions modulo~$p^m$ normalized such that~$y_j^{[m]}(0)=1$. Since~$y_j^{[m+1]}$ is a solution of~$L_jv\equiv0\mod p^{m+1}$, it is also a solution of the same differential equation modulo~$p^m$. In particular, both~$y_j^{[m+1]}$ and~$y_j^{[m]}$ satisfy the recursion in Lemma~\ref{lem:rec} modulo~$p^m$ and their coefficients, up to the power~$t^{\lceil{m/2}\rceil-1}$, are uniquely determined by the recursion and the initial value~$y_{j,0}=1$. It follows that
\[
y_j^{[m+1]}(t)-y_j^{[m]}\equiv O(t^{\lceil{m/2}\rceil})\mod p^m\,,\quad m\ge1\,.
\]
From this and from Lemma~\ref{lem:deg} it follows that the collection~$\{y_j^{[m]}\}_{m\ge1}$ lifts to a solution~$y_j=1+\sum_{n\ge1}y_{j,n}t^n$ of~$L_j=0$ defined over~$\Bigl(\varprojlim{R/p^mR}\Bigr)[\![t]\!]\otimes W(\F)$. Since this holds for every~$p\not\in S$, and since a solution of~$L_jv=0$ has coefficients in~$K$ (as implied by the Frobenius method), it follows that the coefficients of~$y_{j}$ are elements of~$R=\mathcal{O}_K[S^{-1}]$.
\end{proof}

\begin{corollary}
\label{cor:cong}
Let~$Y\hookrightarrow\mathcal{A}_g$ be as in Theorem~\ref{thm:main}, and  let~$N$ be the lowest common multiple of the orders of the elliptic points of~$Y$ (set $N=1$ if~$Y$ has no elliptic points). For~$j\in\{1,\dots,g\}$ denote by~$y_j(t)$ the normalized integral solutions of~$L_jv=0$. Then for every prime~$p\not \in S$ there exists a polynomial~$\alpha_{p,j}(t)\in(\mathcal{O}_K[S^{-1}]\otimes \F_p)[t]$ and an index~$j'\in\{1,\dots,g\}$ such that
\[
y_j(t)^N\;\equiv \alpha_{p,j}(t)\cdot y_{j'}(t^p)^N\mod p\,.
\]
\end{corollary}
\begin{proof}

These congruences follow essentially from the congruence of the matrices~$\{E(\omega,m)\}_m$ described in Lemma~\ref{lem:Em}. 
The integral holomorphic solution~$y_j(t)$ of the Picard-Fuchs operator~$L_j$ is constructed in Theorem~\ref{thm:main} as the inverse limit of approximated solutions~$y_j^{[m]}(t)$. These come in turn from non-zero (modulo~$p$) entries~$\widehat{E}(\omega,m)_{\nu_j,j}(t)$ of the matrix~$\widehat{E}(\omega,m)$ defined in~\eqref{eq:cart}, which satisfy the congruence relation in Lemma~\ref{lem:Em}. 
As explained in the proof of Proposition~\ref{prop:cong}, since~$\{\omega_j\}_j$ is an eigenbasis for the action of real multiplication, the matrix~$\widehat{E}(\omega;m)$ modulo~$p$ has precisely one non-zero entry~$\widehat{E}(\omega,m)_{v_j,j}(t)$ for each row and column. The congruence for the matrices then become congruences between approximated solutions:
\begin{equation}
\label{eq:congas}
y_j^{[m+1]}(t)\equiv\widehat{E}(\omega,m+1)_{v_j,j}(t)\equiv \widehat{E}(\omega,m)_{v_{j'},j'}(t)^p\cdot\widehat{E}(\omega,1)_{v_j,j}(t)\equiv y_{j'}^{[m]}(t^p)\cdot\widehat{E}(\omega,1)_{v_j,j}(t) \mod p\,.
\end{equation}
Here~$j'$ is determined uniquely from~$j$ and the splitting of~$p$ in the field of real multiplication~$F$. Explicitly, it can be computed from the definitions in the statement of Proposition~\ref{prop:cong}.
To prove the corollary, it remains to understand the matrix entry~$\widehat{E}(\omega,1)_{v_j,j}(t)\mod p$, i.e., to understand the action of the Cartier operator. 

Because of the shape of the matrix~$\widehat{E}(\omega,1)$ (see~\eqref{eq:shape} and the discussion after it), every value~$t_0$ such that~$\widehat{E}(\omega,1)_{v_j,j}(t_0)\equiv0\mod p$ is a value for which the matrix~$\widehat{E}(\omega,1)$, i.e., the action of the Cartier operator, is not invertible. 

Given of the relation between the Cartier and the Hasse-Witt matrices (see the appendix of~\cite{K85} for a general result), such values~$t_0$ are related to the points of~$Y$ having non-ordinary reduction modulo~$p$. The number of such points is finite for a given prime~$p$, since we start with a non-trivial family~$\mathcal{X}\to Y$ of abelian varieties over a characteristic-zero field. The full non-ordinary locus is given by all the values that make the Cariter matrix non invertible, while here we focus on the entry~$\widehat{E}(\omega,1)_{v_j,j}(t)\mod p$, whose zero describe a component of the non-ordinary locus of~$Y$ (related to a stratification of the Hilbert modular variety defined by Goren in~\cite{GoHasse}). 

We study~$\widehat{E}(\omega,1)_{v_j,j}(t)\mod p$ locally for~$\fp\subset\mathcal{O}_K$ a prime over~$p$, first for the regular points, i.e., over the $R$-algebra~$B$. In this case the points of non-ordinary reduction on~$R/\fp R$ are described by the zeroes of a polynomial~$g_{j,\fp}(t)\in(R/\fp R)[t]$ (this follows also from Lemma~\ref{lem:Em} and the preceding discussion). 
It is proven in~\cite{AG} that the non-ordinary divisor in a Hilbert modular variety over a perfect field is reduced, which implies that~$g_{j,\fp}(t)$ is simple. The entry~$\widehat{E}(\omega,1)_{v_j,j}$ of~$\widehat{E}(\omega,1) \mod p$ outside the set of regular singular points is given by a tuple~$(g_{j,\fp}(t))_{\fp|p}$ of simple polynomials, whose roots identify the points of non-ordinary reduction. 

We discuss now the situation over the regular singular points, which were removed from the base scheme, but to which the differential operators~$L_j$ extend. The cusps of~$Y$ do not contribute to the non-ordinary locus, as it is known for Hilbert modular varieties that the points in the compactification are never non-ordinary (the Hasse invariant do not vanish at the cups). The contribution of the regular singular point~$t(\tau_0)$ that are not cusps is described by the factor~$(t(\tau_0)-t)^{\epsilon_\fp(j,\tau_0)/n_0}$, where~$\epsilon_\fp(j,\tau_0)\in\{0,1+\mathrm{ord}_{\tau_0}\varphi_j'\}$ is non zero only if~$t(\tau_0)$ is of non-ordinary reduction over~$R/\fp R$, and~$n_0$ is the order of the stabilizer of~$\tau_0$ in~$\Gamma$. This is because the exponent should agree with the characteristic exponents of~$t(\tau_0)$ in the Riemann scheme of Lemma~\ref{lem:Rs}. This factor is again polynomial for points~$\tau_0$ with trivial stabilizer, but an algebraic function for elliptic points, accounting for the non-trivial local monodromy of the differential equation there.

We collect all the polynomial contributions to the non-ordinary locus of~$R/\fp R$ in a single polynomial~$\hat{g}_{j,\fp}(t)$. Then by setting
\begin{equation}
\label{eq:ap}
\alpha_{p,j}(t):=\Bigl(\hat{g}_{j,\fp}(t)^N\cdot\prod_{i=1}^r{(t-t(e_i))^{\tfrac{N}{n_i}\cdot\epsilon_\fp(j,e_i)}}\Bigr)_{\fp|p}\equiv \widehat{E}(\omega,1)_{v_j,j}(t)^N\mod pR,
\end{equation}
where~$N=\mathrm{lcm}(n_1,\dots,n_r)$ and~$N=1$ if there are no elliptic points, and by considering Equation~\eqref{eq:congas}, one obtains the congruences in the statement.  
\end{proof}

\begin{example}
\label{ex:hgD}
The curve~$\Po/\Delta(2,5,\infty)$, where~$\Delta(2,5,\infty)$ is the triangle group with one cusp and elliptic points of order two and five, has a modular embedding into a component of the Hilbert modular surface~$\mathcal{M}_{\Q(\sqrt{5})}(\C)$. The associated Picard-Fuchs differential equations~$L_1$ and~$L_2$ are of hypergeometric type, given by
\[
L_j\=t(1-t)\frac{d^2}{dt^2}\+\Bigl(1-\frac{3}{2}t\Bigr)\frac{d}{dt}\-\frac{25-4j^2}{400}\,,\quad j=1,2\,.
\]
The normalized holomorphic solutions~$y_1(t)$ and~$y_2(t)$ are classical (Gauss's) hypergeometric functions
 \[
 y_j(t)\;:=\;{}_2F_1\Bigl(\frac{5-2j}{20}, \frac{5+2j}{20};1;t\Bigr)\=\sum_{n=0}^\infty\Bigl(\frac{5-2j}{20}\Bigr)_n\Bigl(\frac{5+2j}{20}\Bigr)_n\,\frac{t^n}{(n!)^2}\,,
 \]
where~$(x)_n:=x(x+1)\cdots(x+n-1)$ is Pochhammer's symbol. Let~$p>7$ be a prime. In the case of real quadratic fields (more generally, for Galois extensions), the relation between the indices~$j$ and~$j'$ in~Corollary~\ref{cor:cong} is easy: if~$p$ is split in~$\Q(\sqrt{5})$, then~$j'=j$; if~$p$ is inert in~$\Q(\sqrt{5})$, then~$j'=j+1$, with the convention~$j+1=1$ if~$j=2$. Explicit examples of the congruences predicted by Corollary~\ref{cor:cong} are
\[
y_1(t)^2\equiv (1-t)(3t+1)^2\cdot y_1\bigl(t^{11}\bigr)^2\mod 11\quad\text{and}\quad y_2(t)^2\equiv (1-t)\cdot y_2\bigl(t^{11}\bigr)^2\mod 11
\]
for the split prime~$p=11$, and
\[
y_1(t)^2\equiv (1-t)\cdot y_2\bigl(t^{13}\bigr)^2\mod 13\quad\text{and}\quad y_2(t)^2\equiv (1-t)(3-t)^2\cdot y_1\bigl(t^{13}\bigr)^2\mod 13
\]
for the inert prime~$p=13$.
\end{example}

\section{Integrality of $t$-expansions of modular forms}
In order to translate the integrality of the holomorphic solutions of the Picard-Fuchs differential equations into the integrality of the~$t$-expansion at the cusp of twisted modular forms, it is convenient to work with the larger space of twisted modular forms with multiplier system. If~$v\colon\Gamma\to\C$ is a multiplier system, we denote by~$M_{\vec{k}}(\Gamma,\varphi,v)$ the space of twisted modular forms of weight~$\vec{k}$ and multiplier system~$v$. We moreover consider the (total) space of twisted modular forms
\begin{equation}
\label{eq:defms}
M^t_{(*,\dots,*)}(\Gamma,\varphi)\:=\bigoplus_{\vec{k}\in\Q\times\Z^{g-1}}\bigoplus_{v}{M_{\vec{k}}(\Gamma,\varphi,v)}\,,
\end{equation}
where the inner direct sum is over all multiplier systems~$v$ for~$\Gamma$. Notice that we let the first component of~$\vec{k}$ varies in~$\Q$, but the other components, which are related to the non-trivial modular embeddings, vary over~$\Z$. Denote by~$M_*^t(\Gamma)$ the space of classical modular forms of rational weight and possibly non-trivial multiplier system, obtained by restricting the sum~\eqref{eq:defms} over vectors~$\vec{k}=(k,0,\dots,0)$.

\begin{proposition}
\label{prop:gen}
Let~$\Gamma$ be a genus zero Fuchsian group admitting a modular embedding~$\varphi\colon\Po\to\Po^g$. There exist twisted modular forms~$B_2,\dots,B_g$ of non-parallel weight such that
\[
M^t_{(*,\dots,*)}(\Gamma,\varphi)\simeq M^t_*(\Gamma)[B_2^{\pm1},\dots,B_g^{\pm1}]
\]
is an isomorphism of graded rings. 
\end{proposition}
\begin{proof}
Write~$\varphi_j$ for the~$j$-th component of the modular embedding~$\varphi$ and let~$\varphi_1(\tau)=\tau$. Since~$\Gamma$ has genus zero, for every~$j=2,\dots,g$ it is possible to find a classical modular form~$g_j\in M_*(\Gamma,v_j)$, where~$v_j$ is a multiplier system, such that~$\mathrm{div}(g_j)=\mathrm{div}(\varphi'_j)$ (more on the choice of~$g_j$ later). Let~$\lambda_j\in\Q$ be such that~$\deg(\mathrm{div}(\varphi_j'))=(\lambda_j-1)\cdot\chi(\Po/\Gamma)$, where~$\chi(\Po/\Gamma)$ is the orbifold Euler characteristic of~$\Po/\Gamma$ and set~$\lambda_1=1$. The weight of~$g_j$ is then~$2(1-\lambda_j)$. It follows that the principal branch~$B_j$ of the square root 
\begin{equation}
\label{eq:Bj}
B_j\:=\sqrt{\frac{g_j}{\varphi'_j}}\in M_{(-\lambda_j,0,\dots,0,1_j,0,\dots,0)}(\Gamma,\phi,\sqrt{v_j})
\end{equation}
is holomorphic and has no zeroes in~$\Po\cup\{\mathrm{cusps}\}$. Its weight is~$-\lambda_j$ in the first component, is $1$ in the~$j$-th component, and is~$0$ else. 

Let~$\vec{k}:=(k_1,k_2,\dots,k_g)$ be such that~$k_1\in\Q$ and~$k_2,\dots,k_g\in\Z$, and let~$v$ be a multiplier system for~$\Gamma$. The map
\[
M_{\vec{k}}(\Gamma,\varphi;v)\to M_{\sum_{\lambda_jk_j}}\biggl(\Gamma; v\prod_{j=2}^g{v_j^{-k_j/2}}\biggr)\quad f\mapsto f\prod_{j=2}^g{B_j^{-k_j}}
\]
is an isomorphism of vector spaces, and it induces the isomorphism in the statement. 
\end{proof}

\begin{example}
\label{ex:mft}
For~$\Gamma=\Delta(2,5,\infty)$, the space~$M_{(*,*)}^t(\Gamma,\varphi)$, together with its differential structure, has been determined in~\cite{BN}. It is of the form~$M_{(*,*)}^t(\Gamma,\varphi)=\C[Q_1,Q_2,B]$, where~$Q_1$ and~$Q_2$ are obtained from hypergeometric functions
\[
Q_l(\tau)^3\={}_2F_1\Bigl(\frac{5-2l}{20}, \frac{5+2l}{20};1;t\Bigr)^{l(3+l)}\,,\quad l=1,2\,,
\]
and~$B^2=Q_1/\varphi_1'$. Here~$t$ is a Hauptmodul normalized to have a zero at~$i\infty$, a pole in~$e^{\pi i/5}$, and to take the value~$t(i)=1$.
The multiplier system of a modular form on~$\Delta(2,5,\infty)$ is determined by its values on the generators~$T=\left(\begin{smallmatrix}1 & 2\cos(\pi/5)\\ 0 & 1\end{smallmatrix}\right)$ and~$S=\left(\begin{smallmatrix}0 & 1\\ -1 & 0\end{smallmatrix}\right)$. Let~$v_{Q_l}\colon\Delta(2,5,\infty)\to\C$ be the multiplier system of~$Q_l$. Then~$v_{Q_l}(T)=1$ and~$v_{Q_l}(S)=e^{4\pi i/3}$ for~$l=1,2$.
\end{example}

\begin{lemma}
\label{lem:int1}
The expansion at the cusps~$i\infty$ of the modular forms~$t'$ and~$\varphi_j'$, for~$j=2,\dots,g$, in the parameter~$t$ has coefficients in the ring~$\mathcal{O}_K[S^{-1}]$.
\end{lemma}
\begin{proof}
The statement follows from Theorem~\ref{thm:main} and the structure of the differential equation satisfied by modular forms.
Let~$h_j$ be a modular form of weight~$(0,\dots,0,1,0,\dots,0)$, the non-zero weight being in the~$j$-th entry (one can more generally consider a~$k$-th root of a modular form of weight~$(0,\dots,0,k,0,\dots,0)$). 
For holomorphic functions~$f,g\colon\Po\to\C$ of weight~$k$ and~$l$ respectively let
\[
[f,g]_1\:=kfg'-lf'g\,,\quad [f,g]_2\:=\frac{k(k+1)}{2}fg''-(k+1)(l+1)f'g'+\frac{l(l+1)}{2}f''g\,,\quad `=\frac{d}{d\tau}
\]
be the classical Rankin-Cohen brackets. For~$j=1,\dots,g$ define
\begin{equation}
\label{eq:rc}
A_j(t)\:=\frac{\bigl[(\varphi_j')^{1/2}h_j,t'\bigl]_1}{t'^2(\varphi_j')^{1/2}h_j}\,,\quad B_j(t)\:=\frac{\bigl[(\varphi_j')^{1/2}h_j,(\varphi_j')^{1/2}h_j\bigr]_2}{(t'\varphi_j'^{1/2}h_j)^2}\,.
\end{equation}
Since~$(\varphi_j')^{1/2}h_j$ is of weight~$(1,0,\dots,0)$, it follows from a weight calculation that~$A_j(t)$ and~$B_j(t)$ are modular functions, and in particular algebraic functions of~$t$. A straightforward computation shows that
\begin{equation}
\label{eq:demo}
D_t^2h_j\+ A_j(t)\cdot D_th_j\+ B(t)\cdot h_j\=0\,,\quad D_t:=\frac{1}{t'}\frac{d}{d\tau}\,.
\end{equation}
This gives an explicit description of the differential equation satisfied by the modular form~$h_j$ in terms of Rankin-Cohen brackets. 

Consider now the Picard-Fuchs differential operators~$L_j$ coming from genus zero Kobayashi geodesics. 
They are of modular nature (Section~\ref{sec:modde}), so they are of the form~\eqref{eq:demo} with~$A_j(t)$ and~$B_j(t)$ defined over~$\mathcal{O}_K[S^{-1}](t)$. Theorem~\ref{thm:main} states that the normalized holomorphic solutions at the singular points with local exponents~$(0,0)$ belong to~$\mathcal{O}_K[S^{-1}][\![t]\!]$. These holomorphic solutions of~$L_jv=0$ lift to modular form~$f_j$ (or a root of a modular form) of weight~$(0,\dots,0,1_j,0,\dots,0)$ that satisfy the relations in~\eqref{eq:rc} with~$h_j$ replaced by~$f_j$. The description of~$B_1(t)$, where~$\varphi'_1=1$, immediately implies that~$t'(t)\in\mathcal{O}_K[S^{-1}][\![t]\!]$. Similarly, for~$j\in\{2,\dots,g\}$ the description of~$A_j(t)$ implies that~$(\varphi_j')^{1/2}f_j\in1+t\mathcal{O}_K[S^{-1}][\![t]\!]$. Since also~$f_j(t)\in1+t\mathcal{O}_K[S^{-1}][\![t]\!]$ holds (Theorem~\ref{thm:main}), it follows that~$\varphi_j'(t)\in\mathcal{O}_K[S^{-1}][\![t]\!]$.
\end{proof}

\begin{theorem}
\label{thm:intmf}
Let~$\Gamma$ be a non co-compact genus zero Fuchsian group admitting a modular embedding~$\varphi\colon\Po\to\Po^g$, and suppose that the corresponding Kobayashi geodesic has integral model over~$\mathcal{O}_K[S^{-1}]$. Let~$S_\Gamma$ be the finite set containing~$S$, the prime~$2$, and the primes dividing the order of the elliptic points of~$\Gamma$ and the numerator of the orbifold Euler characteristic~$\chi(\Po/\Gamma)$. For every weight~$\vec{k}=(k_1,k_2,\dots,k_g)\in\Q\times\Z^{g-1}$ there exists a basis of the space~$M_{\vec{k}}(\Gamma,\varphi)$ whose~$t$-expansion at the cusp~$i\infty$ in the parameter~$t$ has coefficients in~$\mathcal{O}_K[S_\Gamma^{-1}]$.
\end{theorem}
\begin{proof}
For a genus zero Fuchsian group it is possible to construct all modular forms as algebraic functions of a Hauptmodul~$t$ and its derivative~$t'$. Assume that~$\Gamma$ has cusps~$c_1=i\infty,c_2,\dots,c_s$ and elliptic points~$\tau_1,\dots,\tau_r\in\Po$ of orders~$e_1,\dots, e_r$ respectively, and assume that~$r,s\ge1$. Set~$e:=\prod_i{e_i}$ and let~$-\chi(\Gamma)=s-2+\sum_{i}(1-1/e_i)\in\tfrac{1}{e}\Z$ be the orbifold characteristic of~$\Po/\Gamma$. Normalize~$t$ to have a zero at~$i\infty$ and a pole at~$\tau_1$. Modular forms~$Q_{1},\dots,Q_{r}$ on~$\Gamma$ with minimal divisor
\begin{equation}
\label{eq:div}
\mathrm{div}Q_{i}\=\frac{1}{e_i}\cdot\tau_i\,,\quad i=1,\dots,r\,,
\end{equation} and possibly non-trivial mulitplier system are given by:
\[
Q_{i}\=
\begin{cases}
\left[\Bigl(\frac{t'}{t}\Bigr)^e\prod_{l=2}^r{\bigl(t-t(\tau_l)\bigr)^{-e(1-1/e_l)}}\prod_{m=2}^s\bigl(t-t(c_m)\bigr)^{-e}\right]^{\frac{1}{e\cdot e_1\cdot(-\chi(\Gamma))}}&\;i=1\,,\\
\left[\Bigl(\frac{t'}{t}\Bigr)^e\bigl(t-t(\tau_i)\bigr)^{e\cdot(-\chi(\Gamma))}\prod_{l=2}^r{\bigl(t-t(\tau_l)\bigr)^{-e(1-1/e_l)}}\prod_{m=2}^s\bigl(t-t(c_m)\bigr)^{-e}\right]^{\frac{1}{e\cdot e_i\cdot(-\chi(\Gamma))}}&\;i\ge2\,.
\end{cases}
\]
Every modular form on~$\Gamma$ is a polynomial in~$Q_1,\dots,Q_r$. A similar construction is possible in the torsion-free case, where all modular forms are of the form~$f^k\cdot p(t)$ where~$f$ is a modular form of weight one with divisor supported at the cusp where~$t$ has its unique pole, and~$p(t)$ is a suitable polynomial.

Lemma~\ref{lem:int1} states that there is a choice of~$t$ such that~$t'$ has $t$-expansion at~$i\infty$ with coefficients in~$\mathcal{O}_K[S^{-1}]$. It follows that for this choice of~$t$, the modular form~$Q_i$ has~$t$-expansion at~$i\infty$ in~$\mathcal{O}_K[S_\Gamma^{-1}][\![t]\!]$ for every~$i=1,\dots,r$ . Consequently, for every~$k\in\Q$, the space of modular forms~$M_k(\Gamma)$ has a basis  with~$t$-expansion at~$i\infty$ in~$\mathcal{O}_K[S_\Gamma^{-1}][\![t]|\!]$. 

Finally, consider general weight~$\vec{k}\in\Q\times\Z^{g-1}$. 
Given the~$S_\Gamma$-integrality of classical modular forms and the~$S$-integrality of~$(\varphi'_j)^{1/2}$ for every~$j$ (Lemma~\ref{lem:int1}), it follows that the modular forms~$B_2,\dots,B_g$ in~\eqref{eq:Bj} can be chosen to have $t$-expansion in~$\mathcal{O}_K[S_\Gamma^{-1}][\![t]|\!]$. The result follows from the isomorphism of Proposition~\ref{prop:gen}.
\end{proof}

\begin{remark}
\label{rmk:cong}
The congruences for solutions of Picard-Fuchs differential equations in Corollary~\ref{cor:cong} translate immediately to congruences between~$t$-expansions of twisted modular forms. This is the content of the Corollary in the introduction. 
In general these congruences take place between modular forms with non-trivial multiplier systems. However, by taking appropriate powers, these can always be restated for modular forms with trivial multiplier system. Here we make this explicit for the case~$g=2$; the general case works in the same way, by considering the more general splitting behavior of a prime in the field of real multiplication~$F$.

Let~$y_1(t)$ and~$y_2(t)$ be the integral holomorphic solutions of the Picard-Fuchs differential equations, let~$f_1(\tau)=y_1(t)^{k_1}$ and~$f_2(\tau)=y_2(t)^{k_2}$ be their lift to twisted modular forms on~$\Gamma$ with trivial multiplier system, and let~$M=\mathrm{lmc}(k_1,k_2)$. Let~$p\not\in S$ be a prime, and let~$N$ and~$\alpha_{p,j}(t)$, for~$j=1,2$, be as in Corollary~\ref{cor:cong}. Then, for~$j=1,2$, it holds
\[
f_j(t)^{\frac{N\cdot M}{k_j}}\;\equiv\; \begin{cases}
\alpha_{p,j}(t)^M\cdot f_{j+1}(t^p)^{\frac{N\cdot M}{k_{j+1}}}& \text{if }p\text{ is inert in }F\\
\alpha_{p,j}(t)^M\cdot f_{j}(t^p)^{\frac{N\cdot M}{k_j}}& \text{if }p\text{ is split in }F
\end{cases}\mod p\,,
\]
with the convention~${j+1}=2$ if~$j=2$\,.
\end{remark}

\section{Non-ordinary locus from truncated solutions of Picard-Fuchs equations}
\label{sec:nolocus}
Consider a genus zero Kobayashi geodesic~$Y\hookrightarrow\mathcal{A}_g$, and let~$N$ be the lowest common multiple of the order of the elliptic points of~$Y$. Consider the model of~$Y$ over~$\mathcal{O}_K[S^{-1}]$, and let~$L_1,\dots,L_g$ be the associated Picard-Fuchs differential operators. 
In Corollary~\ref{cor:cong} it has been proven that, for every~$j\in\{1,\dots,g\}$, the normalized integral solution~$y_j(t)$ satisfies the congruence, for~$p\not\in S$, 
\begin{equation}
\label{eq:apj}
y_j(t)^N\equiv\alpha_{p,j}(t)y_{j'}(t^p)^N\mod p\,,
\end{equation}
for some~$j'\in\{1,\dots,g\}$ and polynomial~$\alpha_{p,j}(t)\in(\mathcal{O}_K[S^{-1}]\otimes\F_p)[t]$ as in~\eqref{eq:ap}. As the proof of that corollary explains, the zeros of~$\alpha_{p,j}(t)$ are related to the fibers of~$Y$ having non-ordinary reduction modulo~$p$. Since fibers over cusps have ordinary reduction, a Hauptmodul~$J$ with pole at a cusp should be considered in the study of the non-ordinary locus, so that to every point of non-ordinary reduction~$\tau_0$ corresponds to a finite value~$J(\tau_0)$. This is the choice we make for the next theorem. 

In the proof of Proposition~\ref{prop:gen} we defined rational numbers~$\lambda_1=1,\lambda_2,\dots,\lambda_g$, called~\emph{Lyapunov exponents}, by the property
\begin{equation}
\label{eq:lj}
\deg(\mathrm{div}(\varphi_j'))=(\lambda_j-1)\chi(Y)\,,
\end{equation}
where~$\varphi_j$ is the~$j$-th component of the modular embedding of~$Y$, and~$\chi(Y)$ is the orbifold Euler characteristic of~$Y$. 

\begin{theorem}
\label{thm:deg}
Let~$Y\hookrightarrow\mathcal{A}_g$ be a genus zero affine Kobayashi geodesic with integral model over~$\mathcal{O}_K[S^{-1}]$ and let~$p\not\in S$ be a prime. Let~$J$ be a Hauptmodul for~$\Gamma$, where~$Y\simeq \Po/\Gamma$, normalized to have the pole at a cusp. 
For every~$j\in\{1,\dots,g\}$ let~$\alpha_{p,j}(J)\in(\mathcal{O}_K[S^{-1}]\otimes\F_p)[J]$ be the polynomial in~\eqref{eq:apj} corresponding to the choice of Hauptmodul~$J$, and let~$\lambda_1,\dots,\lambda_g$ be the Lyapunov exponents in~\eqref{eq:lj}.  
Then
\[
\deg{\alpha_{p,j}}(J)= \frac{-\chi(Y)}{2}\cdot N\cdot(p\lambda_{j'}-\lambda_j)\,,
\]
where~$j$ and~$j'$ are related as in~\eqref{eq:apj}.
\end{theorem}
\begin{proof}
Let~$(s^\infty_j,s_j^\infty)$ be the local exponents of the singular point at infinity (corresponding to the cusp where~$J$ has the pole) for the differential operator~$L_j$. Let~$x=1/J$ be a local coordinate at the singular point at infinity, and let~$y_j^\infty(x)=x^{s^\infty_j}\cdot(1+\cdots)$ be the normalized solution at~$x=0$ obtained from Frobenius method. In particular, the coefficients of the power series~$y_j^\infty(x)\cdot x^{-s^\infty_j}=1+\cdots$ satisfy the recursion coming from the differential equation. 

The same strategy used to prove Theorem~\ref{thm:main} and~Corollary~\ref{cor:cong} shows that~$y_j^\infty(x)\in\mathcal{O}_K[S^{-1}][\![x]\!]$ for every~$j\in\{1,\dots,g\}$ (because the local exponents are the same), and that for every~$j$ there exists~$j'$ such that
\begin{equation}
\label{eq:solinf}
y^\infty_j(x)^N\;\equiv\;\alpha_{p,j}(\tfrac{1}{x})\cdot y^\infty_{j'}(x^p)^N\mod p\,,
\end{equation}
This holds because~$E(\eta;1)_{v,j}(1/x)$, from which~$\alpha_{p,j}$ is obtained by rising to the~$N$-th power, is a solution of~$L_j\mod p$ in~$x=0$, as a straightforward verification shows. 

By comparing the exponent of~$x$ in~\eqref{eq:solinf}, it follows that~$\deg{\alpha_{p,j}}(J)=N(p s^\infty_{j'}-s_j^{\infty})$, so it is left to express the local exponents~$s^\infty_j$ and~$s^\infty_{j'}$ in terms of invariants of~$Y$ and of the modular embedding. This can be done by using Fuchs's relation (see Chapter 2 of~\cite{Yo}) which claims that the sum of the local exponents of~$L_j$ is equal to~$m_j-1$, where~$m_j$ is the number of finite singular points of~$L_j$. 
Because of the modular nature of the differential operator~$L_j$, it holds
\begin{equation}
\label{eq:mj}
m_j-1\=\sharp\{\text{cusps}\}+\sharp\{\text{ell. points}\}- 2 +r_j= 2g-2 +\sharp\{\text{cusps}\}+\sharp\{\text{ell. points}\}+r_j\,,
\end{equation}
since~$g=0$, where~$r_j$ is the number of singular points of~$L_j$ that do not correspond to cusps nor to elliptic points.
The application of Fuchs's relation to the Riemann scheme of~$L_j$ in Lemma~\ref{lem:Rs} gives then
\begin{align*}
2s^\infty_j&\=m_j-1-\sum_{i=s+1}^{m_j}{\frac{1+\mathrm{ord}_{\tau_i}(\varphi_j')}{e_i}}\=m_j-1-\sum_{i=s+1}^{m_j}\frac{1}{e_i}-\deg{\mathrm{div}(\varphi_j')}\\
&\= \chi(Y)+r_j-\sum_{i=m_j-r_j+1}^{m_j}\frac{1}{e_i}-\deg{\mathrm{div}(\varphi_j')}\=\chi(Y)-\deg{\mathrm{div}(\varphi_j')}\=\lambda_j\chi(Y)\,,
\end{align*}
where the second identity holds because the divisor of~$\varphi_j'$ is not supported at the cusps, the third identity holds because~$m_j-1-\sum_{i=s+1}^{m_j-r_j}{1/e_i}=\chi(Y)$, the fourth holds since~$e_i=1$ for~$i=m_j-r_j+1,\dots,m_j$, and the last one is a consequence of~\eqref{eq:lj}.
\end{proof}

\begin{remark}
\label{rmk:Haupt}
For a Hauptmodul~$t$ with a pole in~$\Po$, the degree of the associated polynomial~$\alpha_{p,j}(t)$ is in general smaller than the one in Theorem~\ref{thm:deg}. The reason is that the fiber over the pole of~$t$, even if of non-ordinary reduction, do not appear as a root of the polynomial~$\alpha_{p,j}(t)$, since~$t$ has not finite value at that point. 

For instance, consider~$\Gamma=\SL_2(\Z)$ and let~$J$ the usual~$j$-function with a pole at~$i\infty$. For~$p=11$ we have~$\alpha_{11}(J)=J^3(1728-J)^2$, since the supersingular elliptic curves for~$p=11$ are the ones with $J=0$ and~$J=1728$ with multiplicities~$1/2$ and~$1/3$ respectively. Consider now the Hauptmodul~$t=1/J$, which has a zero at~$i\infty$ and a pole where~$J=0$ Then~$\alpha_{11}(t)=t^5t^{-3}(1728-t^{-1})^2=(1728\cdot t-1)^2$ has degree two, since the supersingular fiber over~$J=0$ can not be seen by~$t$.
We will consider again this fact in Example~\ref{ex:deltat}. 
\end{remark}

\begin{corollary}
\label{cor:calc}
Let~$Y\hookrightarrow\mathcal{A}_g$ be a genus zero affine Kobayashi curve defined over~$\mathcal{O}_K[S^{-1}]$, and assume that it has at least two cusps (see Remark~\ref{rem:cusp} for the general case). Let~$t$ be a Haputmodul with the zero and the pole at a cusp, and let~$N$ be the lowest common multiple of the order of the elliptic points of~$Y$ (set~$N=1$ if there are no elliptic points).
Let~$y_1(t),\dots,y_g(t)$ be the normalized holomorphic solution of the Picard-Fuchs equation~$L_jv=0$ in~$t=0$ for~$j\in\{1,\dots,g\}$. Let~$p\not\in S$ be a prime, and let~$\alpha_{p,j}(t)$ be as in~\eqref{eq:ap}.  
Assume that~$\deg{\alpha_{p,j}(t)}<p$ for every~$j\in\{1,\dots,g\}$. The zeros of the polynomial 
\begin{equation}
\label{eq:nop}
\mathrm{no}_p(t)=\mathrm{lcm}_{j=1,\dots,g}\Bigl(\bigl[y_j(t)^N\bigr]_p\Bigr)\mod p
\end{equation}
where~$[\sum_{n=0}^\infty{b_n}t^n]_m:=\sum_{n=0}^{m-1}{b_nt^n}$ denotes the truncation of the power series at order~$m$, 
identifies the points of non-ordinary reduction of~$Y\hookrightarrow\mathcal{A}_g$. Similarly, the zeroes of
\begin{equation}
\label{eq:spp}
\mathrm{sp}_p(t)=\mathrm{gcd}_{j=1,\dots,g}\Bigl(\bigl[y_j(t)^N\bigr]_p\Bigr)\mod p
\end{equation}
identifies the points of superspecial reduction. 

In the special case~$g=2$, the zeroes of the polynomials~\eqref{eq:nop} and~\eqref{eq:spp} describe the supersingular locus in the case~$p$ is inert or split in the field of real multiplication respectively. 
\end{corollary}
\begin{remark}
\label{rem:cusp}
It is important that~$t$ is zero at a cusp so that Theorem~\ref{thm:main} can be applied to give solutions with integral coefficients. The assumption that~$t$ has the pole at a cusp is to prevent that~$t$ has a pole at a points of non-ordinary reduction, since they are never cusps. In the case the curve has only one cusp (for instance a triangle curve), Corollary~\ref{cor:calc} still holds (same proof) but only gives the non-ordinary locus outside the point where~$t$ has its pole. In Example~\ref{ex:deltat} below it is explained how to use Corollary~\ref{cor:calc} to recover the non-ordinary locus completely in the case of triangle curves, but the idea works more generally. Moreover, as shown in the same example, in the presence of elliptic points one can lower the number~$N$ by placing the pole of~$t$ at an elliptic point, enlarging then the class of curves the for which Corollary~\ref{cor:calc} applies.
\end{remark}

\begin{proof}
The relation between the zeroes of the polynomials~$\alpha_{p,j}$ and the non-ordinary locus of~$Y\hookrightarrow\mathcal{A}_g$ has been explained in the proof of Corollary~\ref{cor:cong}. Under the assumption of the theorem, Corollary~\ref{cor:cong} implies that the polynomial~$\alpha_{p,j}$ can be obtained by truncating $y_j(t)^N$ at order~$p$ and reducing it modulo~$p$. Since the union of the zeros of all~$a_{p,j}(t)$ gives the non-ordinary locus of~$Y$, this proves~\eqref{eq:nop}. The superspecial locus is the locus of abelian surfaces that are isomorphic to a product of supersingular elliptic curves, The superspecial points in~$\mathcal{M}_F(\mathbb{F})$ are the intersection points of the components of the non-ordinary locus (see~\cite{AG}). Its restriction to the reduction of curve~$Y$ is the set of common zeros of the polynomials~$\alpha_{p,j}(t)$. This proves~\eqref{eq:spp}.

If~$g=2$, Bachmat-Goren~\cite{BG} prove that the non-ordinary locus coincide with the supersingular locus if~$p$ is inert in the field of real multiplication, and that the supersingular locus coincides with the superspecial locus if~$p$ is split. This proves the last statement.
\end{proof}

\begin{corollary}
\label{cor:size}
Let~$Y\hookrightarrow\mathcal{A}_g$ be a genus zero affine Kobayashi curve with~$r$ elliptic points, and consider its integral model~$\mathcal{X}\to Y$ defined over~$\mathcal{O}_K[S^{-1}]$. Let~$p\not\in S$, and let~$n_p$ be the number of prime ideals of~$\mathcal{O}_K$ over~$p$. Denote by~$\sharp\mathrm{no}_p^Y, \sharp\mathrm{ss}_p^Y, \sharp\mathrm{sp}_p^Y$ the cardinality of the non-ordinary locus, supersingular locus, superspecial locus  of~$\mathcal{X}\to Y$ modulo~$p$ respectively. 
\begin{enumerate}
\item The non-ordinary locus of~${Y}$ modulo~$p$ has cardinality bounded by
\[
\Bigl\lfloor\frac{-\chi(Y)}{2}\cdot\max_{(j,j')}\{p\lambda_{j'}-\lambda_j\}\Bigr\rfloor\;\le\; \sharp\mathrm{no}_p^Y\;\le\;n_p\cdot\sum_{(j,j')}\Bigl\lfloor\frac{-\chi(Y)}{2}{(p\lambda_{j'}-\lambda_{j})}\Bigr\rfloor+r\,,
\]
where~$j$ and~$j'$ are related as in Corollary~\ref{cor:cong}, and~$\lfloor{x}\rfloor$ denotes the integral part of~$x\in\Q$.
\item Let~$g=2$, and let~$\lambda_2\le\lambda_1=1$ be the Lyapunov exponent associated to~$Y\hookrightarrow\mathcal{A}_2$. If~$p\not \in S$ is inert in~$F$, then 
\[
\Bigl\lfloor{(p-\lambda_2)\frac{-\chi(Y)}{2}}\Bigr\rfloor\;\le\;\sharp\mathrm{ss}_p^Y\;\le\;n_p\cdot\Bigl\lfloor{(p-1)(1+\lambda_2)\frac{-\chi(Y)}{2}}\Bigr\rfloor\+r\,\,
\]
and~$\sharp\mathrm{sp}_p^Y\le \lfloor{(p-\lambda_2)\frac{-\chi(Y)}{2}}\rfloor+r$.

If~$p$ is split in~$F$, then
\[
\Bigl\lfloor{(p-1)\frac{-\chi(Y)}{2}}\Bigr\rfloor\;\le\;\sharp\mathrm{no}_p^Y\;\le\;n_p\cdot\Bigl\lfloor{(p-1)(1+\lambda_2)\frac{-\chi(Y)}{2}}\Bigr\rfloor\+r\,\,
\]
and~$\sharp\mathrm{ss}_p^Y\le \lfloor{\lambda_2(p-1)\frac{-\chi(Y)}{2}}\rfloor+r$.
\end{enumerate}
\end{corollary}
\begin{proof}
The size of the non-ordinary locus of~$\mathcal{X}\to Y$ modulo~$p$ is bounded by the sum of the sizes of its components, indexed by~$j\in\{1,\dots,g\}$. The size of each component is bounded by the number of roots of the polynomial~$\alpha_{p,j}(J)$. 
In order to estimate this number, let~$\fp\subset\mathcal{O}_K[S^{-1}]$ be a prime over~$p$, and consider the non-ordinary points of $Y$ mod~$\fp$ outside the elliptic locus, that is, the roots of the polynomial~$\widehat{g}_{j,\fp}$ in~\eqref{eq:ap}. Theorem~\ref{thm:deg} implies that~$\deg(g_{j,\fp})\le \lfloor{\frac{-\chi(Y)}{2}(p\lambda_{j'}-\lambda_j)}\rfloor$. The size of the~$j$-th component of the non-ordinary locus of~$Y$ mod~$p$ is bounded by $n_p\cdot\lfloor{\frac{-\chi(Y)}{2}(p\lambda_{j'}-\lambda_j)}\rfloor+r$, since for every prime~$\fp$ above~$p$ the curve $Y\mod\fp$ has at most~$\deg(\widehat{g}_{j,\fp})$ non-ordinary points outside the elliptic locus, and at most~$r$ non-ordinary points coming from the~$r$ elliptic points. The upper bound in the statement is finally obtained by summing over all components. 

The lower bound corresponds to the case where the non-ordinary points in every residue field come from the same points of~$Y$, and all elliptic points have ordinary reduction. 

The statement for the~$g=2$ case follows from the general bound by noticing that~$\lambda_2\le\lambda_1=1$ and~$j'=j$ if~$p$ is split in~$F$, and~$j'=j+1$ is~$p$ is inert (with the convention that~$j+1=1$ if~$j=2$). It is then easy to compute the lower bound in Part one.  
The specialization to the supersingular locus follows from a result of Bachmat and Goren~\cite{BG}, who prove that for~$g=2$ the supersingular locus and the non-ordinary locus coincide if~$p$ is inert in~$F$, and that the supersingular locus coincide with the superspecial locus if~$p$ is split in~$F$ (recall that in any case the superspecial locus is given by the intersection of the components of the non-ordinary locus).
\end{proof}

\begin{remark}
The upper bound in Corollary~\ref{cor:size} is sharp. This follows immediately for~$K=\Q$ and~$Y$ torsion free, but it is true also in non-torsion-free cases. For instance consider~$Y\simeq\Po/\SL_2(\Z)$ considered as the~$J$-line, where~$J$ is the classical~$j$-invariant, and the prime~$p=11$. In this case~$n_p=1$, $\lambda_{j'}=\lambda_j=1$ and~$r=2$. The upper bound implies that there are at most~$\lfloor{\tfrac{10}{12}}\rfloor+2=2$ supersingular elliptic curves, corresponding to the elliptic points. From the classical result of Deruring, we know that the upper bound is attained. 
\end{remark}

\begin{example}[Torsion free modular curves]
\label{ex:mocu}
If~$Y$ has no elliptic points, by Theorem~\ref{thm:deg} it is possible to obtain the non-ordinary locus via truncation of solutions only if~$Y$ has at most four cusps. This is the case of Igusa's original discovery~\cite{IgLeg} for the Legendre family, where the base curve~$Y\simeq\Po/\Gamma(2)$ has three cusps and no torsion. 
According to Sebbar's classification~\cite{Se} of torsion-free genus zero congruence subgroups of~$\SL_2(\R)$, there are eight congruence subgroups that satisfy the assumption of Corollary~\ref{cor:calc} (some of them are related to Zagier's search~\cite{ZagierA}):
\[
\Gamma(2),\, \Gamma_0(4),\,\Gamma(3),\, \Gamma_0(4)\cap\Gamma(2),\,\Gamma_1(5),\,\Gamma_0(6),\,\Gamma_0(8),\,\Gamma_0(9)\,.
\]
As an example, consider the differential equation for~$\Gamma_1(5)$, which is related to Apéry's proof of irrationality of~$\zeta(2)$:
\[
t(t^2+11t-1)y''(t)\+(3t^2+22t-1)y'(t)+(t+3)y(t)\=0\,.
\] 
There is an explicit description of the normalized holomorphic solution~$y(t)=1+\sum_{n\ge1}{a_nt^n}$ in terms of binomial coefficients~$a_n=\sum_{k=0}^n{\binom{n}{k}^2\binom{n+k}{k}}$. This description, together with Corollary~\ref{cor:calc} and Beukers's interpretation of Apery's differential equation as a Picard-Fuchs differential equation for family of elliptic curves~\cite{BeIrr}, prove that the supersingular locus of the family
\[
y^2\=x^3+\frac{(t^4+6t+1)}{4}x^2+\frac{t(t+1)}{2}x+\frac{t^2}{4},\quad t\in\C\smallsetminus\biggl\{0,\frac{-11}{2}\pm\frac{5\sqrt{5}}{2}\biggr\}\,
\]
is given by the zeros of the simple polynomial
\[
1\+\sum_{n=1}^{p-1}\biggl(\sum_{k=0}^n{\binom{n}{k}^2\binom{n+k}{k}}\biggr)t^n\mod p\,.
\]
\end{example}

\begin{example}[Triangle curves]
\label{ex:deltat}
Let~$\Delta(n,m,\infty)$ for~$n\le m<\infty$ be a Fuchsian group with one cusp and elliptic points~$e_n$ and~$e_m$ of order~$n$ and~$m$ respectively. The Picard-Fuchs differential equations in this case are hypergeometric with rational parameters. The polynomial~$\alpha_{p,j}(J)$ of~Theorem~\ref{thm:deg} has the form, for~$N=\mathrm{lcm}(n,m)$,
\begin{equation}
\label{eq:apjt}
\alpha_{p,j}(J)=g_{p,j}(J)^N\cdot(J(e_n)-J)^{\epsilon_{p,n}\frac{N}{n}}\cdot  (J(e_m)-J)^{\epsilon_{p,m}\frac{N}{m}}\,.
\end{equation}
Let the indices~$j$ and~$j'$ be related as in Theorem~\ref{thm:deg}. Since~$-\chi(\Po/\Delta(n,m,\infty))=(mn-m-n)/(mn)$, by using Theorem~\ref{thm:deg} we see that 
\begin{equation}
\label{eq:dega}
\mathrm{deg}(\alpha_{p,j}(J))\=\frac{mn-m-n}{2mn}\cdot N\cdot (p\lambda_{j'}-\lambda_j)
\end{equation}
is in general not smaller than~$p$, as required by Corollary~\ref{cor:calc} (take for instance~$m$ and~$n$ coprime). This comes from the choice of Hauptmodul with pole at the unique cusp. 

Let us instead consider a Hauptmodul~$t$ with a pole at the elliptic point~$e_m$ of order~$m$ and a zero at the cusp. As in Remark~\ref{rmk:Haupt}, since~$t$ has not a finite value at~$e_m$, the factor~$t-t(e_m)$ can not appear in the in the solution in~$t=0$ of the Picard-Fuchs differential equations. It follows that we can replace~$N=\mathrm{lcm}(n,m)$ with~$N=n$ and the congruences~\eqref{eq:apj} for the holomorphic solutions of the Picard-Fuchs differential equation become
\[
y_j(t)^n\equiv g_{p,j}(t)^n\cdot (t(e_n)-t)^{\epsilon_{p,n}}\cdot y_{j'}(t)^{n\cdot p}\mod p,\quad j,j'\in\{1,\dots,g\}\,,
\]
where~$g_{p,j}(t)$, and~$\epsilon_{p,n}$ are as in~\eqref{eq:apjt}.   
In other words, we lose the description of~$e_m$ as a point of non-ordinary reduction, but the degree of the polynomial involved in the congruence above is smaller than the degree of~$\alpha_{p,j}(J)$; it follows in fact from the above computation of~$\mathrm{deg}(\alpha_{p,j})$ that
\begin{equation}
\label{eq:sizet}
\begin{aligned}
\mathrm{deg}( g_{p,j}(t)^n\cdot (t(e_n)-t)^\epsilon_{p})&\=\frac{-\chi(\Po/\Delta(n,m,\infty))}{2}\cdot n\cdot (p\lambda_{j'}-\lambda_j)\-\frac{n\cdot\epsilon_{p,m}}{m}\\
&\=\frac{(mn-m-n)}{2m}\cdot (p\lambda_{j'}-\lambda_j)\-\frac{n\cdot\epsilon_{p,m}}{m}\,.
\end{aligned}
\end{equation}
Since~$\lambda_j\le 1$ for every~$j=1,\dots,g$ and~$m\ge n$, the first summand above is smaller than~$p$ for infinitely many choices of~$(n,m)$ with~$n\le m$, for instance for all choices with~$n\le 4$ and~$m\ge n$ arbitrary, which include the family of Hecke groups~$\Delta(2,m,\infty)$. For this choice of Hauptmodul, a version of Corollary~\ref{cor:calc} tells that by truncation of (a power of) the holomorphic solutions of the Picard-Fuchs equations (in $t$) we can determine the non-ordinary locus outside the elliptic point~$e_m$. 

In order to complete the study of the non-ordinary locus, we recover~$\alpha_{p,j}(J)$ by multiplying a power of the truncated solution by a suitable power of~$J$
\[
\alpha_{p,j}(J)\=J^{\mathrm{deg}(\alpha_{p,j})}\cdot\Bigl(\bigl[y_j(t)^n\bigr]_p\Bigl)^m\,,\quad t=\frac{1}{J}
\] 
It finally follows, similarly to~Corollary~\ref{cor:calc}, that the zeroes of
\begin{equation}
\label{eq:degt}
\mathrm{lcm}_{j=1,\dots,g}\Bigl\{J^{\mathrm{deg}(\alpha_{p,j})}\cdot\Bigl(\bigl[y_j(t)^n\bigr]_p\Bigr)^m\Bigr\}\,,
\end{equation}
correspond to the fibers of non-ordinary reduction of the integral model of~$\Po/\Delta(n,m,\infty)$. Th degree~$\mathrm{deg}(\alpha_{p,j})$ is explicit and computed in~\eqref{eq:dega}. By replacing~$\mathrm{lcm}$ with~$\mathrm{gcd}$ one finds the superspecial locus.

As a concrete example, we determine the supersingular locus of the curve~$\Po/\Delta(2,5,\infty)$, whose two Picard-Fuchs equations are described in~Example~\ref{ex:hgD}. The associated family of abelian surfaces is explicitly described in~\eqref{eq:ext} and is embedded in the Hilbert modular variety of real multiplication field~$F=\mathbb{Q}(\sqrt{5})$. An prime~$p\in\Z$ unramified in~$F$ can only be inert or (totally) split in~$F$. In particular, the following holds for the indices~$j,j'\in\{1,2\}$:
\[
j'=\begin{cases}
2 & \text{if }j=1\\
1 & \text{if }j=2
\end{cases}\quad \text{ for $p$ inert,}\quad
j'=\begin{cases}
1 & \text{if }j=1\\
2 & \text{if }j=2
\end{cases}\quad \text{ for $p$ split.}
\]
The Lyapunov exponents for this triangle curve are~$\lambda_1=1$ and~$\lambda_2=\frac{1}{3}$, and the orbifold Euler characteristic is~$\chi(\Po/\Delta(2,5,\infty))=-3/10$. Let~$J$ be a Hauptmodul with a pole at~$i\infty$, and such that~$J(e_2)=1$ and~$J(e_5)=0$, and let~$t:=J^{-1}$. Then
\[
\mathrm{deg}(\alpha_{p,1})\=\frac{3}{20}\cdot 10\cdot\Bigl(\frac{p}{3}-1\Bigr)\=\frac{p-3}{2}\,,\quad \mathrm{deg}(\alpha_{p,2})\=\frac{3}{20}\cdot 10\cdot\Bigl(p-\frac{1}{3}\Bigr)\=\frac{3p-1}{2}
\]
if~$p$ is inert in~$F$, and
\[
\mathrm{deg}(\alpha_{p,1})\=\frac{3}{20}\cdot 10\cdot(p-1)\=\frac{3(p-1)}{2}\,,\quad \mathrm{deg}(\alpha_{p,2})\=\frac{3}{20}\cdot 10\cdot\Bigl(\frac{p}{3}-\frac{1}{3}\Bigr)\=\frac{p-1}{2}
\]
if~$p$ is split in~$F$. It follows from~\eqref{eq:degt} and Bachmat-Goren's result~\cite{BG} the supersingular locus of~$\Po/\Delta(2,5,\infty)$ modulo~$p$ is given by

\[
\mathrm{lcm}\biggl(J^{\frac{p-3}{2}}\biggr[{}_2F_1\Bigl(\frac{3}{20}, \frac{7}{20};1;\frac{1}{J}\Bigr)^2\biggl]^5_p,J^{\frac{3p-1}{2}}\biggl[{}_2F_1\Bigl(\frac{1}{20}, \frac{9}{20};1;\frac{1}{J}\Bigr)^2\biggr]^5_p\biggr)\mod p
\]
if~$p$ is inert in~$F$, and 
\[
\mathrm{lcm}\biggl(J^{\frac{3(p-1)}{2}}\biggr[{}_2F_1\Bigl(\frac{3}{20}, \frac{7}{20};1;\frac{1}{J}\Bigr)^2\biggl]^5_p,J^{\frac{p-1}{2}}\biggl[{}_2F_1\Bigl(\frac{1}{20}, \frac{9}{20};1;\frac{1}{J}\Bigr)^2\biggr]^5_p\biggr)\mod p
\]
if~$p$ is split in~$F$.
\end{example}

\section*{Appendix. Families of Jacobians of genus two curves and absolute Igusa invariants}
Jacobians of genus 2 curves are classified, over every field, by the value of three invariants~$(J_1,J_2,J_3)$ called~\emph{absolute Igusa invariants}~\cite{Ig}. It is therefore natural and desirable to describe the non-ordinary, supersingular, and superspecial locus of a family of Jacobians of genus two curves in terms of these invariants. The general idea is simple. 

The graded ring of holomorphic Siegel modular forms for~$\mathrm{Sp}_2(\Z)$ is~$\C[\psi_4,\psi_6,\chi_{10},\chi_{12}]$, where~$\psi_k$ is the normalized Siegel Eisenstein series of weight~$k\ge4$ and~$\chi_{10}$ and~$\chi_{12}$ are Siegel cusp forms of weight~$10$ and~$12$ respectively. The absolute Igusa invariants are the generators of the function field of the Siegel modular variety, and can be written as rational functions of the above Siegel modular forms. A choice is (see~\cite{LY})
\begin{equation}
\label{eq:Ig}
J_1\:=2\cdot3^5\frac{\chi_{12}^5}{\chi_{10}^6}\,\quad J_2\:=2^{-3}\cdot3^3\frac{\psi_4\chi_{12}^3}{\chi_{10}^4}\,,\quad J_3\:=2^{-5}\cdot3\biggl(\frac{\psi_6\chi_{12}^2}{\chi_{10}^3}+2^3\cdot3\frac{\psi_4\chi_{12}^3}{\chi_{10}^4}\biggr)\,.
\end{equation}
For an affine Kobayashi curve~$Y\hookrightarrow\mathcal{M}_F(\C)$ in a Hilbert modular surfaces we have a chain of embeddings
\[
Y\hookrightarrow\mathcal{M}_F(\C)\hookrightarrow\mathcal{A}_2\,,
\]
where~$\mathcal{A}_2$ is the Siegel modular threefold. One can first pull-back the Igusa invariants along the map~$\mathcal{M}_F(\C)\hookrightarrow\mathcal{A}_2$, and then pull them back to~$Y$ by using the modular embedding maps. Since the dimension of the function fields has dropped from three to one, on~$Y$ the Igusa invariants are no longer algebraically independent; the genus zero assumption implies moreover that every Igusa invariant is a rational function of a Hauptmodul~$J$ of~$Y$. 
It finally follows that the values of the absolute Igusa invariants of the non-ordinary, supersingular, and superspecial reduction of~$Y$ modulo~$p$ can be computed from the non-ordinary, supersingular, and superspecial polynomial in~$J$ respectively.
We work out the details for the curve~$\Po/\Delta(2,5,\infty)\hookrightarrow X^-_5:=\Po^2/\SL_2(\mathcal{O}_5^\vee\oplus\mathcal{O}_5)$. 

\begin{lemma} 
\label{lem:Igusa}
Let~$J$ be a Hauptmodul for~$\Delta(2,5,\infty)$ normalized to have a pole at the cusp infinity and~$J(e^{2\pi i/5})=0$ and~$J(i)=1$. 
Then the restriction of the absolute Igusa invariants~$J_1,J_2,J_3$ on~$\Po/\Delta(2,5,\infty)$ is given by
\[
J_1\Bigl|_{W_5}\=2^{-7}7^5J^2\,,\quad J_2\Bigl|_{W_5}\=2^{-7}7^3J^2\,,\quad J_3\Bigl|_{W_5}=2^{-7}3^{-1}7^2J(5J-2^43^3\sqrt{5})\,.
\]
\end{lemma}
\begin{proof}
Resnikoff~\cite{Re} computed the pull-back of the holomorphic Siegel modular forms to the Hilbert modular surface~$X_5:=\Po/\SL_2(\mathcal{O}_5)$. It involves only symmetric Hilbert modular forms of even parallel weight, whose space had been previously determined by Gundlach~\cite{Gu}:
\[
M_{2*}^{\text{sym}}(\SL_2(\mathcal{O}_5))=[g_2,s_6,s_{10}]\,,
\]
where~~$g_k$ denotes the Eisenstein series of parallel weight~$k$ whose~$q$-expansion at~$\infty$ has constant term~$1$, and
\begin{align*}
s_6&=\frac{67}{2^5\cdot3^3\cdot 5^2}\bigl(g_2^3-g_6\bigr) \,,\\
s_{10}&=\frac{1}{2^{10}\cdot3^5\cdot5^5\cdot7}\bigl(2^2\cdot3\cdot7\cdot4231 g_2^5-5\cdot67\cdot2293g_2^2g_6+412751g_{10}\bigr) 
\end{align*}
are cusp forms of weight~$6$ and~$10$ respectively. Resinkoff's result is then
\begin{equation}
\label{eq:res}
\psi_4\bigl|_{X_5}\= g_2^2\,,\quad\psi_6\bigl|_{X_5}\=g_2^3-2^53^3s_6\,,\quad \chi_{10}\bigl|_{X_5}\=\frac{s_{10}}{4}\,,\quad
\chi_{12}\bigl|_{X_5}\=\frac{3s_6^2-2g_2s_{10}}{12}\,.
\end{equation}
We have to compute the pull-back of the Hilbert modular forms on~$X_5$ to the curve~$W_5$. One has to be careful here, since~$W_5$ lies not in~$X_5$ but in a different Hilbert modular surface~$X_5^-$. Luckily, these are isomorphic, the isomorphism being the map~$X_5\to X^-_5$ given by~$(z_1,z_2)\mapsto(\varepsilon z_1,\varepsilon^\sigma z_2)$, where~$\epsilon\in\mathcal{O}_5$ is a unit of negative norm. One has to compose the modular embedding of~$W_5$ in~$X_5^-$ with this isomorphism to get the correct restriction of Hilbert modular forms. The restriction is computed by using the method of~\cite{MZ}, Section~8.1, which we do not repeat here. We use moreover the fact that the space of twisted modular forms for~$W_5$ are the polynomial ring~$\C[f_1,s_5]$, where~$f_1$ is of parallel weight~$1$ and~$s_5$ is a cusp form of parallel weight~$5$ (see the appendix of~\cite{BN} for a proof). Moreover the Hauptmodul~$J$ is given by~$f_1^5/s_5$. It follows that 
\[
g_2\bigl|_{W_5}=f_1^2\,,\quad s_6\bigl|_{W_5}=\frac{\sqrt{5}}{2}f_1s_5\,,\quad s_{10}\bigl|_{W_5}=s_5^2\,.
\]
By combining these identities with~\eqref{eq:res} and~\eqref{eq:Ig}, one gets the statement.
\end{proof}
Thanks to lemma, for a given~$p$ one can compute the absolute Igusa invariants of the supersingular fibers of the family over~$\Po/\Delta(2,5,\infty)$ from the supersingular values of the Hauptmodul~$J$. As an example, from the computations of the previous section we find the following supersingular values for~$p>7$.
\begin{table}[h!]
\[
\begin{array}{c|cccc}
p & 11 & 13 & 17 & 19 \\
\hline 
(J_1,J_2,J_3)& (3,5,8+4\sqrt{5}) & (0,0,0) & (0,0,0) & (0,0,0) \\
& & (1,4,9+5\sqrt{5}) & (5,6,16+15\sqrt{5}) & (13,15,9(1+\sqrt{5})) \\  
& & (3,12,1+6\sqrt{5}) & (3,7,13(1+\sqrt{5})) &  \\
 & & & (6,14,9+7\sqrt{5}) &  \\
\end{array}\,.
\]
\end{table}
\bibliography{Int}{}
\bibliographystyle{plain}

\bigskip
\noindent
UNIVERSITÄT BIELEFELD, GERMANY\\
\textit{Email address}: \texttt{gbogo@math.uni-bielefeld.de}

\end{document}